\documentclass[12pt]{amsart}
\usepackage{amsmath,bbm, color}
\usepackage{latexsym,mathrsfs,bm}
\usepackage{url}
\usepackage[english]{babel}
\usepackage{paralist}
\usepackage{amsthm,amsfonts,amssymb,amsmath}
\usepackage[latin1]{inputenc}

\newcommand{\tRe}{\textup{Re }}
\newcommand{\tIm}{\textup{Im }}

\newcommand{\sumstar}{\sideset{}{^*}\sum}
\newcommand{\sumh}{\sideset{}{^h}\sum}
\newcommand{\sumd}{\sideset{}{^d}\sum}

\newcommand{\sumt}{\sideset{}{^{rel}}\sum}

\newcommand{\es}[1]{\begin{equation}\begin{split}#1\end{split}\end{equation}}
\newcommand{\est}[1]{\begin{equation*}\begin{split}#1\end{split}\end{equation*}}

\renewcommand{\mod}[1]{~\pr{\textnormal{mod}~#1}}
\newtheorem{thm}{Theorem}[section]

\newtheorem{prop}[thm]{Proposition}
\newtheorem{lem}[thm]{Lemma}
\newtheorem{lemma}[thm]{Lemma}

\theoremstyle{remark}

\newtheorem{rem*}{Remark}
\newcommand{\pr}[1]{\left( #1\right)}
\newcommand{\pg}[1]{\left\{ #1\right\}}

\newcommand{\e}[1]{\operatorname{e}\pr{ #1}}

\def\sumstar{\operatornamewithlimits{\sum\nolimits^*}}

\newcommand{\sumtwo}{\operatorname*{\sum\sum}}

\newcommand{\sumfour}{\operatorname*{\sum\sum\sum\sum}}

\newcommand{\summany}{\operatorname*{\sum... \sum}}
\newcommand{\comment}[1]{}

\let\originalleft\left
\let\originalright\right
\renewcommand{\left}{\mathopen{}\mathclose\bgroup\originalleft}
\renewcommand{\right}{\aftergroup\egroup\originalright}

\numberwithin{equation}{section}

\begin{document}
\title{The eighth moment of the family of $\Gamma_1(q)$-automorphic $L$-functions}

\date{\today}

\author[V. Chandee]{Vorrapan Chandee}
\address{Department of Mathematics \\ Burapha University \\ 169 Long-Hard Bangsaen rd, Saensuk, Mueng, Chonburi, Thailand,  20131}
\email{vorrapan@buu.ac.th }

\author[X. Li]{Xiannan Li}
\address{Department of Mathematics \\ Kansas State University \\
	138 Cardwell Hall, Manhattan, KS 66506, United States}
\email{xiannan@math.ksu.edu}

\allowdisplaybreaks
\numberwithin{equation}{section}
\selectlanguage{english}
\begin{abstract}
We prove a Lindel\"of on average bound for the eighth moment of a family of $L$-functions attached to automorphic forms on $GL(2)$, the first time this has been accomplished.  Previously, such a bound had been proven  for the sixth moment for our family by Djankovi\'c \cite{Dj} and for a similar family by Young \cite{Young}.  Our proof rests on a new approach which overcomes the lack of perfect orthogonality in the family initially observed by Iwaniec and Li \cite{IL}.
\end{abstract}

\maketitle  
\section{Introduction}   
One particularly prominent field of study in analytic number theory is the size and distribution of values of $L$-functions.  These are studied not only for its own sake, but also because good understanding of these properties frequently imply great results about arithmetic objects attached to these $L$-functions.  The larger program here is to understand $L$-functions in general, and thereby understand many dissimilar arithmetic objects via the same route.  However, it turns out that this is quite challenging, and different families of $L$-functions require very different methods, despite their superficial similarities. 

In this paper, we study a family of $L$-functions attached to automorphic forms on $GL(2)$.  In particular, we shall derive an upper bound for the eighth moment of this family, which is of the same quality as the Lindelof hypothesis on average.  Previously, Djankovic \cite{Dj} had derived the same upper bound for the sixth moment of this family.  A comparable family had been studied by M. Young \cite{Young} and Luo \cite{Luo}.  Essentially, they study the family of $L$-functions attached to Hecke Maass forms with Laplace eigenvalue $\lambda_j = 1/4 + t_j^2$, averaged over all $t_j \in [T, 2T]$.  Young \cite{Young} then derived an upper bound for the sixth moment of this family of the same quality as the Lindelof hypothesis on average.  However, the best upper bound for the eighth moment due to \cite{Luo} exceeds the Lindelof quality bound by $T^{1/2}$.  Thus, our work is the first time a Lindelof on average bound has been achieved for the eighth moment of a $GL(2)$ family.

Typically upper bounds of this type are closely connected to large sieve type bounds.  As noted by other authors, such bounds are more subtle for many $GL(2)$ families as compared to the classical $GL(1)$ family with Dirichlet characters.  For our family, a large sieve was developed by Iwaniec and Xiaoqing Li \cite{IL}.  However, their large sieve indicates that the family is not perfectly orthogonal - in particular, certain terms are larger than expected if the coefficients are chosen to look like certain Bessel functions twisted by Kloosterman sums.  It is for this structural reason that Djankovi\'{c} was able to develop good upper bounds for the sixth moment, and his methods fail in the case of the eighth moment.

Now, we will be more precise.  Let $S_k(\Gamma_0(q), \chi)$ be the space of cusp forms of weight $k \ge 2$ for the group $\Gamma_0(q)$ and the nebentypus character $\chi \mod q$,  where $$\Gamma_0(q) = \left\{ \left( \left.{\begin{array}{cc}
   a & b \\
   c & d \\
  \end{array} } \right)  \ \right|  \  ad- bc = 1 , \ \  c \equiv 0 \mod q \right\}.$$ 
Also, let $S_k(\Gamma_1(q))$ be the space of holomorphic cusp forms for the group 
$$\Gamma_1(q) = \left\{ \left( \left.{\begin{array}{cc}
   a & b \\
   c & d \\
  \end{array} } \right)  \ \right|  \  ad- bc = 1 , \ \  c \equiv 0 \mod q, \ \ \ a \equiv d \equiv 1 \mod q \right\}.$$
Note that $S_k(\Gamma_1(q))$ is a Hilbert space with the Petersson's inner product
$$ <f, g> = \int_{\Gamma_1(q) \backslash \mathbb H} f(z)\bar{g}(z) y^{k-2} \> dx \> dy,$$
and
$$ S_k(\Gamma_1(q)) = \bigoplus_{\chi \mod q} S_k(\Gamma_0(q), \chi).$$  

Let $\mathcal H_\chi \subset  S_k(\Gamma_0(q), \chi)$ be an orthogonal basis of $ S_k(\Gamma_0(q), \chi)$ consisting of Hecke cusp forms, normalized so that the first Fourier coefficient is $1$. For each $f \in \mathcal H_\chi $, we let $L(f, s)$ be the $L$-function associated to $f$, defined for $\tRe(s) > 1$ as 
\es{\label{def:Lfnc} 
L(f,s) = \sum_{n \geq 1} \frac {\lambda_f(n)}{n^{s}} = \prod_p \pr{1 - \frac{\lambda_f(p)}{p^s} + \frac{\chi(p)}{p^{2s}}}^{-1},}
where $\{\lambda_f(n)\}$ are the Hecke eigenvalues of $f$.  With our normalization, $\lambda_f(1) = 1 $.  In general, the Hecke eigenvalues satisfy the Hecke relation
\es{\label{eqn:Heckerelation}
\lambda_f(m) \lambda_f(n) = \sum_{d|(m,n)} \chi(d) \lambda_f\pr{\frac{mn}{d^2}},} 
for all $m,n \geq 1$. 
We define the completed $L$-function as
\es{\label{def:completedLfnc}\Lambda\pr{f,  s} = \pr{\frac q{4\pi^2}}^{\frac s2} \Gamma \pr{s + \tfrac {k-1}2} L\pr{f, s},}
which satisfies the functional equation
\es{\label{eqn:fnceq}\Lambda\pr{f, s} = i^k \overline{\eta}_f\Lambda\pr{\bar{f}, 1 - s},}
where $|\eta_f| = 1$ when $f$ is a newform.  

Suppose for each $f \in \mathcal H_\chi$, we have an associated number $\alpha_f$. Then we define the harmonic average of $\alpha_f$ over $\mathcal H_\chi$ to be 
\est{\sumh_{f \in \mathcal H_\chi} \alpha_f = \frac{\Gamma(k-1)}{(4\pi)^{k-1}}\sum_{f \in \mathcal H_\chi} \frac{\alpha_f}{\|f\|^2}.} 

We shall be interested in moments of the form 

\est{M_k(q) := \frac{2}{\phi(q)}\sum_{\substack{\chi \mod q \\ \chi(-1) = (-1)^k}} \sumh_{f \in \mathcal H_\chi} |L(f, 1/2)|^{2k}.}

 We note that the size of the family is around size $q^2$ while the conductor is around size $q$.  This should be compared with the family previously mentioned in the work of \cite{Young} and Luo \cite{Luo} where the family is around size $T^2$ with conductor around size $T$. \footnote{The mechanism is somewhat different however - this family exhibits a certain conductor dropping phenomenon while ours has increased size.} For prime level, $\eta_f$ can be expressed in terms of Gauss sums, and in particular we expect $\eta_f$ to equidistribute on the circle as $f$ varies over an orthogonal basis of $S_k(\Gamma_1(q))$.  Thus, we expect our family of $L$-functions to be unitary.
 
As mentioned previously, Djankovi\'{c} \cite{Dj} studied the sixth moment of this family and obtained the following upper bound, which is consistent with the Lindel\"{o}f hypothesis.
 
 \begin{thm}[Djankovi\'{c}] Let $q$ be a prime number and $k \geq 3$ an odd integer. Then we have as $q \rightarrow \infty$
 	$$ M_3(q) \ll q^\epsilon$$
 	for any $\epsilon > 0$. Note that the implied constant depends on $\epsilon$.
 	
 \end{thm}
 
In recent work \cite{CL2}, the authors were able to derive an asymptotic with a power saving for the sixth moment.  In this paper, we prove the following Lindel\"{o}f upper bound for the eighth moment.
 \begin{thm} \label{mainthm:bound} Let $q$ be a prime number and $k \geq 3$ an odd integer. Then we have as $q \rightarrow \infty$
 	$$ M_4(q) \ll q^\epsilon$$
 	for any $\epsilon > 0$. Note that the implied constant depends on $\epsilon$.
 	
 \end{thm}

Our methods may extend to help prove an asymptotic on a comparable eighth moment, perhaps involving an extra small average over the vertical line.  We hope to return to this in the future.


\section{Approximate functional equation and initial setup} \label{sec:approx}

The first step is to express $|L(f, 1/2)|^8$ in terms of an approximate functional equation.  To this end, we introduce the following notations and lemmas.

 The $k$-divisor function is defined by
 \es{\label{def:sigma_k} \sigma_k(n) = \sum_{n_1n_2...n_k = n} 1 . }
 When $k = 2$, we write $\sigma(n)$ instead of $\sigma_2(n)$.  $\sigma$ is a multiplicative function, but is  not completely multiplicative.  The following Lemma records the well known multiplicative relation for $\sigma$.

 \begin{lemma} \label{lem:multofsigma_k}
 	We have
 	\est{\sigma(n_1n_2) = \sum_{d| (n_1, n_2)} \mu(d) \sigma \pr{\frac{n_1}{d}}\sigma\pr{\frac{n_2}{d}}.}
 \end{lemma}
 
 Let 
 \es{ \label{eqn:defB}\mathscr A(b, c; d, j) := \mu(b)\mu(c) \sigma \pr{\frac{d[b, c]}{b}}\sigma \pr{\frac{d[b, c]}{c}} \sigma(j).}  Now we write $L^4(f, s)$ in terms of its Dirichlet series.

 \begin{lem} \label{eqn:prod3Lfnc} Let $L(f, s)$ be an $L$-function in $\mathcal H_\chi.$  For $\tRe(s) > 1$, we have
 	\est{ L^4\pr{f, s} &= \sum_{j \geq 1} \frac{\chi(j) \sigma(j)}{j^{2s}} \sum_{d \geq 1} \frac{\chi(d)}{d^{2s}} \sum_{n \geq 1} \frac{\lambda_f(n)}{n^s} \sum_{n_1n_2 = u} \sigma(dn_1) \sigma(dn_2) \\
 		&=  \sumfour_{b, c, d, j \geq 1} \frac{\chi(j[b, c]d)\mathscr A(b, c; d, j)}{(bc)^{s} [b, c]^{2s}d^{2s}j^{2s}}   \sum_{n \geq 1} \frac{\lambda_f(bcn) \sigma_4(n)}{n^{s}}.}
 \end{lem}
 
 \begin{proof}
 	
 	From the Hecke relation (\ref{eqn:Heckerelation}), we have
 	\est{&L^2\pr{f, s} = \sum_{j \geq 1} \frac{\chi(j)}{j^{2s}}  \sum_{n\geq 1} \frac{\lambda_f \left(n\right) \sigma(n)}{n^{ s}}.}
 	Then by the Hecke relation, we obtain
 	\est{ 
 		L^4\pr{f, s} &= \sum_{j \geq 1} \frac{\chi(j) \sigma(j)}{j^{2s}}  \sumtwo_{n_1, n_2\geq 1} \frac{\lambda_f \left(n_1\right) \lambda_f(n_2) \sigma(n_1) \sigma(n_2)}{n_1^{ s}n_2^s} \\
 		&= \sum_{j \geq 1} \frac{\chi(j) \sigma(j)}{j^{2s}} \sum_{d \geq 1} \frac{\chi(d)}{d^{2s}} \sum_{u \geq 1} \frac{\lambda_f(n)}{n^s} \sum_{n_1n_2 = n} \sigma(dn_1) \sigma(dn_2).}
 	From Lemma \ref{lem:multofsigma_k}, we have
 	\est{&\sum_{d \geq 1} \frac{\chi(d)}{d^{2s}} \sum_{n \geq 1} \frac{\lambda_f(n)}{n^s} \sum_{n_1n_2 = n} \sigma(dn_1) \sigma(dn_2) \\
 		&= \sum_{d \geq 1} \frac{\chi(d)}{d^{2s}} \sumtwo_{n_1, n_2 \geq 1} \frac{\lambda_f(n_1n_2)}{n_1^s n_2^s} \sumtwo_{b| (n_1, d), \  c | (n_2, d)} \mu(b) \mu(c)\sigma\pr{\frac{n_1}{b}} \sigma\pr{\frac{d}{b}}\sigma\pr{\frac{n_2}{c}}\sigma\pr{\frac{d}{c}} \\
 		&= \sumtwo_{b, c \geq 1} \mu(b)\mu(c) \sum_{\substack{d \geq 1 \\ [b, c] | d}}  \frac{\chi(d)}{d^{2s}} \sumtwo_{\substack{n_1, n_2 \geq 1 \\ b | n_1, c | n_2}} \frac{\lambda_f(n_1n_2)}{n_1^s n_2^s} \sigma\pr{\frac{n_1}{b}} \sigma\pr{\frac{d}{b}}\sigma\pr{\frac{n_2}{c}}\sigma\pr{\frac{d}{c}} \\
 	}
 	By changing $n_1$ to $bn_1$, $n_2$ to $cn_2$ and $d$ to $d[b,c]$, we derive the lemma.
 \end{proof}

 Let $H(s)$ be an even holomorphic function bounded in the region $|\textrm{Re}(s)| < 3$ with $H(0) = 1$. We now prove an approximate functional equation for $L^4(f, 1/2)$.

 \begin{lemma} \label{lem:approxFncLmodular} For any $\xi> 0,$ let
 	\est{\mathcal V(\xi) = \frac{1}{2\pi i} \int_{(1)} H^4\pr{s} \frac{\Gamma^4\pr{\tfrac k2 + s}}{\Gamma^4\pr{\tfrac k2}} \xi^{-s} \ \frac{ds}{s}.}
 	Then  
 	\es  {\label{eqn:approxfcneq} &L(f, 1/2)^4 = \summany_{b, c, d, j, n \geq 1} \frac{\chi(j[b, c]d)\mathscr A(b, c; d, j) \lambda_f(bcn) \sigma_4(n) }{j[b, c]d \sqrt{bcn} }  \mathcal V\pr{\frac{(2\pi)^4 (j[b, c] d)^2 bcn}{q^2}}\\
 		& \ \ \ \ + (i^k \bar \eta_f)^4\summany_{b, c, d, j, n  \geq 1} \frac{\overline{\chi}(j[b, c]d)\mathscr A(b, c; d, j) \overline{\lambda_f}(bcn) \sigma_4(n)}{ j[b, c]d \sqrt{bcn}}  \mathcal V\pr{\frac{(2\pi)^4 (j [b, c] d)^2 bcn}{q^2}}.}
 \end{lemma}
 
 \begin{proof} We start by writing
 	$$\mathcal J(f) = \frac{1}{2\pi i } \int_{(2)} \Lambda\pr{f, \tfrac 12 + s}^4 H^4(s) \> \frac{ds}{s}, $$
 	where we recall that $\Lambda(f,s)$ is defined in (\ref{def:completedLfnc}). Moving the contour integral to $(-2)$ and using the functional equation (\ref{eqn:fnceq}) and $H(-s) = H(s)$, we obtain that 
 	\est{\mathcal J(f) &= \Lambda\pr{f, \tfrac 12}^4  + \frac{1}{2\pi i}\int_{(-2)}\Lambda\pr{f, \tfrac 12 + s}^4  H^4(s) \> \frac{ds}{s} \\
 		&= \Lambda\pr{f, \tfrac 12}^4  - \frac{1}{2\pi i}\int_{(2)}\Lambda\pr{f, \tfrac 12 - s}^4  H^4(-s)  \> \frac{ds}{s} \\
 		&= \Lambda\pr{f, \tfrac 12}^4 - (i^k \bar \eta_f)^4 \mathcal J(\bar f).}
 	By \eqref{def:completedLfnc} and Lemma \ref{eqn:prod3Lfnc}, we have 
 	\est{\mathcal J(f) = \pr{\frac{q}{4\pi^2}}\Gamma^4\pr{\frac k2}\summany_{b, c, d, j, n \geq 1} \frac{\chi(j[b, c]d)\mathscr A(b, c; d, j) \lambda_f(bcn) \sigma_4(n) }{j[b, c]d \sqrt{bcn} }  \mathcal V\pr{\frac{(2\pi)^4 (j[b, c] d)^2 bcn}{q^2}},} and a similar expression holds for $\mathcal J(\bar f)$. 	Using the expression for $\Lambda(f, 1/2)$ from Equation (\ref{def:completedLfnc}) completes the proof of the Lemma.
 	
 \end{proof}

By Stirling's formula, for any fixed $ A > 0$ and real number $x > 0$
$$ \mathcal V(x) \ll \frac{1}{(x + 1)^A}   $$
and 
$$ \mathcal V(x) = 1 + O(x^A)   \ \ \ \textrm{for}  \ \ \ x \rightarrow 0.$$  
This implies that the main contribution of $L(f,1/2)^4$ comes from when $(j[b, c]d)^2 ghn \ll q^{2 + \epsilon_0}.$  
Let $\Psi$ be supported in the interval $[1,2]$ such that  $x^\nu \Psi^{(\nu)}(x) \ll 1$ for all $\nu \leq 0$, and let $\Psi_1(x) = \Psi(x)\mathcal V\left( \frac{x X}{q^{2 + \epsilon}}\right)$. Hence $\Psi_1(x)$ has the same properties as $\Psi(x)$ for all $X \leq q^{2 + \epsilon}.$  By Lemma \ref{lem:approxFncLmodular}, it is enough to bound 
\est{\frac{2}{\phi(q)}\sum_{\substack{\chi \mod q \\ \chi(-1) = (-1)^k}} \sumh_{f \in \mathcal H_\chi} |\mathcal T_f(X)|^2,}
where 
\est{\mathcal T_f(X) = \summany_{\substack{b, c, d, j, n \geq 1 \\ (j[b, c]d)^2 bcn \ll q^{2 + \epsilon}}} \frac{\chi(j[b, c]d)\mathscr A(b, c; d, j) \lambda_f(bcn) \sigma_4(n) }{j[b, c]d \sqrt{bcn}}  \Psi_1\pr{\frac{(2\pi)^4 (j[b, c] d)^2 bcn}{X}}.}
Applying Cauchy-Schwarz inequality and writing $[b, c] = bc/(b, c)$, we obtain that
\est{|\mathcal T_f(X)|^2 \leq & \summany_{\substack{b, c, d, j \geq 1 \\ j [g, h]d \leq \sqrt X}} \frac{(b, c) \sigma_2^2(j)\sigma^2_2\left( \frac{d[b, c]}{b} \right)\sigma^2_2\left( \frac{d[b, c]}{c} \right)}{jdbc} \summany_{ \substack{b, c, d, j \geq 1 \\ j [b, c]d \leq \sqrt X}} \frac{(b, c)}{jdbc} \\
	& \hskip 1in \times \left|\sum_{n \geq 1} \frac{\lambda_f(bcn) \sigma_4(n)}{\sqrt{bcn} } \Psi_1\pr{\frac{(2\pi)^4 (j[b, c] d)^2 bcn}{X}} \right|^2}
Since  
\est{\summany_{\substack{b, c, d, j \geq 1 \\ j [b, c]d \leq \sqrt X}} \frac{(b, c) \sigma_2^2(j)\sigma^2_2\left( \frac{d[b, c]}{b} \right)\sigma^2_2\left( \frac{d[b, c]}{c} \right)}{jdbc} \ll X^\epsilon \ll q^\epsilon,}
and 
\est{\summany_{ \substack{b, c, d, j \geq 1 \\ j [b, c]d \leq \sqrt X}} \frac{(b, c)}{jdbc} \ll q^\epsilon,}
to prove Theorem \ref{mainthm:bound}, it suffices to show the following Proposition.

\begin{prop}\label{prop:mainbound} Let $b, c, n$ be defined as above, and $$a_{bcn} = \frac{\sigma_4(n) }{\sqrt{bcn}} \Psi_1 \left( \frac{bcn}{Y}\right),$$
where $Y = \frac{X}{16\pi^4 (j [b, c]d)^2 }$.
 Then we have
\est{  \frac{2}{\phi(q)}\sum_{\substack{\chi \mod q \\ \chi(-1) = -1}} \sumh_{f \in \mathcal H_\chi} \left|\sum_{n \geq 1} a_{bcn}\lambda_f(bcn)  \right|^2\ll q^{\epsilon}.}
Note that the implied constant depends on $\epsilon.$
\end{prop}  

\section{Preliminary lemmas}\label{sec:prelemma}

To deal with the summation in Proposition \ref{prop:mainbound}, we first apply the large sieve developed by Iwaniec and Xiaoqing Li \cite{IL} which we record below.

\begin{lem}[Asymptotic large sieve] \label{lem:ALS} Let $q$ be a prime number, $N \geq q$, $T = Nq^{-1}$ and $1 \leq H \leq T.$ Then for any complex vectors $\alpha = (a_n)$ with $N < n \leq 2N$ we have
	\est{\frac{2}{\phi(q)}\sum_{\substack{\chi \mod q \\ \chi(-1) = (-1)^k}} \sumh_{f \in \mathcal H_\chi} \left|\sum_{n \geq 1} a_n \lambda_f(n) \right|^2 = \frac{1}{q} \sum_{\substack{1 \leq t \leq T \\ (t, q) = 1}} \left( \frac{2\pi}{t}\right)^2 \sum_{1 \leq h \leq H} |P_{ht}(\alpha)|^2 + O \pr{N^\epsilon \pr{\frac N{q^2} + \sqrt{\frac{N}{qH}} }} \| \alpha \|^2}
with any $\epsilon > 0$,
$$ P_{ht}(\alpha) =  \sum_n a_n S(h\bar q, n ; t) J_{k - 1} \pr{\frac{4\pi }{t} \sqrt{\frac{hn}{q}}},$$
$S(m, n, c)$ is the Kloosterman sum defined by 
$$ S(m, n; c) = \sum_{a\bar a \equiv 1 \mod c} \e{\frac{am + \bar a n}{c}},$$
and $J_{k - 1}(x)$ is the $J$-Bessel function of order $k - 1.$ Moreover, the imiplied constant depends on $k$ and $\epsilon.$	
\end{lem}


Next, we provide some useful properties of Bessel functions. 
\begin{lem} \label{lem:Besselresult}
	We have
	\es{\label{asympJxbig} J_{k-1} (2\pi x) = \frac{1}{2\pi\sqrt x}\pg{ W(2\pi x)\e{x - \frac k4 + \frac 18}  +  \overline{W}(2\pi x)\e{-x + \frac k4 - \frac 18}},}
	where $W^{(j)}(x) \ll_{j, k} x^{-j} $.  Moreover, 
	\es{\label{asympJxSm} J_{k-1} (2 x) = \sum_{\ell = 0}^{\infty} (-1)^{\ell} \frac{x^{2\ell + k - 1}}{\ell! (\ell + k - 1)!}.}

\end{lem}
These results are standard and we refer the reader to \cite{Watt} for these claims. Other than $J$-Bessel functions, we will also need to use some properties of a Bessel function of the second kind $Y_0(x)$ and a modified Bessel function of the second kind $K_0(x)$. These appear due to an application of Voronoi summation with coefficients $\tau(n)$, which we will see later.  We express $Y_0(x)$ and $K_0(x)$ in three different forms. The first and second forms are useful for large $x$ and small $x$ respectively, and the third expression is helpful in separating variables.

\begin{lem} \label{lem:Y0K0}
Let $Y_0(x)$ be the Bessel function of the second kind of order 0, and $K_0(x)$ is the modified Bessel function of the second kind of order 0.  Then
	
	 \es{\label{asymY0} Y_{0} (2\pi x) = \frac{1}{\pi\sqrt x} \textrm{Im}\pr{ W(2\pi x)\e{x - \frac 18}  },}
	 and 
	 \es{\label{asymK0} K_0(x) = \pr{\frac{1}{2x}}^{1/2} e^{-x} \, W_1(x), }
	 where $W^{(j)}(x) \ll_{j, k} x^{-j} $ and $W_1^{(j)}(x) \ll_j x^{-j}$. This $W$-function is the same as the one in Lemma \ref{lem:Besselresult}. Moroever, 
	 \es{\label{Y0taylor} Y_{0}(x) = \frac{2}{\pi} \pr{\ln \pr{\frac x2} + \gamma} \sum_{k = 0}^\infty  (-1)^k \frac{x^{2k}}{4^k(k !)^2}  + \frac{2}{\pi}\sum_{k = 1}^\infty (-1)^{k + 1} H_k \frac{x^{2k}}{4^k (k!)^2},}
	 and
	 \es{\label{K0taylor}   K_{0}(x) = -\pr{\ln \pr{\frac x2} + \gamma} \sum_{k = 0}^\infty   \frac{x^{2k}}{4^k(k !)^2}  + \sum_{k = 1}^\infty  H_k \frac{x^{2k}}{4^k (k!)^2},}
where $\gamma$ is the Euler's constant and $H_k$ is a harmonic number, defined by
$$ H_k = \sum_{m = 1}^k \frac{1}{m}.$$

Let $0 < \sigma < 1 $. Then for some constants $\kappa$ and $\kappa_1$
	 \es{\label{Y0int}  Y_0(x) = \frac{1}{2\pi i}\int_{(-\sigma)} \gamma(s) x^{-s} \> ds + \frac 2{\pi}\ln x + \kappa,}
	 and
	 \es{\label{K0int} K_0(x) = \frac{1}{2\pi i} \int_{(-\sigma)} \gamma_1(s) x^{-s} \> ds - \ln x + \kappa_1 }
	 for some $\gamma(s) $ and $\gamma_1(s)$ satisfying $\int_{(-\sigma)} |\gamma(s)| \> |ds| \ll 1$ and $\int_{(-\sigma)} |\gamma_1(s)| \> |ds| \ll 1$.
	 
	
\end{lem}

\begin{proof} The result of (\ref{asymY0}) and (\ref{asymK0}) can be found on p.206 in \cite{Watt}, and Equation (\ref{Y0taylor}) and (\ref{K0taylor}) are given in Section 9 of \cite{AS}. We are left to prove Equation (\ref{Y0int}) and (\ref{K0int}). From the Mellin transform in Equation 17 of Section 6.8  \cite{table}, we obtain that
	$$ Y_0(x) = \frac{1}{2\pi i}\int_{(\sigma)} \gamma(s) x^{-s} \> ds; \ \ \ \ \ \ 0 < \sigma  < 3/2$$
	where 
	$$ \gamma(s) = -2^{s-1} \pi^{-1} \Gamma^2\pr{\frac s2} \cos\pr{\frac{s\pi}{2}}.  $$ 
	
	The integral representation above is only conditionally convergent, where the limit is taken for $-T < \tIm s < T$ and letting $T$ go to infinity.  By a standard argument, we may shift the contour of integration to $-\sigma$ where $0<\sigma < 1$.  In so doing, we pick up the residue of $\gamma(s) x^{-s}$ at $s=0$ which is $\frac 2{\pi}\ln x - \frac 2{\pi}\ln 2 + \frac{2 \gamma}{\pi}$ using the Laurent expansion of $\Gamma(s)$ at $s=0$.  Thus
	\begin{equation*}
	Y_0(x)=\frac{1}{2\pi i}\int_{(-\sigma)} \gamma(s) x^{-s} \> ds + \frac 2{\pi}\ln x - \frac 2{\pi}\ln 2 + \frac{2 \gamma}{\pi},
	\end{equation*}
	where $\gamma$ is  the Euler's constant.  By Stirling's formula, the integral above is now absolutely convergent.  In particular, 
	$$|\gamma(s)| \ll |t|^{-\sigma - 1},
	$$for $s = -\sigma + it$ and $|t| >1$. This proves Equation (\ref{Y0int}).
	
	Finally, Equation (\ref{K0int}) follows from the Mellin transform in Equation 26 of Section 6.8 \cite{table}, which is
	$$K_0(x) = \frac{1}{2\pi i} \int_{(\sigma)} \gamma_1(s) x^{-s} \> ds ;\ \ \ \ \ \ \ \ \ \  \sigma > 0$$
	where $\gamma_1 (s) = 2^{s - 2} \Gamma^2\pr{\frac s2}. $ We then shift the contour of integration to $-\sigma$ where $0<\sigma < 1$ and pick up the residue of $\gamma_1(s) x^{-s}$ at $s=0$, which is $ \ln2 - \ln x -\gamma. $ This implies Equation (\ref{K0int}).
\end{proof}

Our summation will involve divisor functions, and we will apply Voronoi summation formula (e.g. see Theorem 4.10  in \cite{IK}).

\begin{lem}[Voronoi summation formula] \label{lem:voronoidn}
	Suppose $g(x)$ is smooth and compactly supported on $\mathbb R^+$. $Y_0$ and $K_0$ are Bessel functions defined as in Lemma \ref{lem:Y0K0}. Let $ad \equiv 1 \mod c$. 
	 Then
\est{\sum_{n = 1}^\infty \sigma(n) \e{\frac{an}{c}} g(n) &= \frac{1}{c} \int_0^\infty (\log x + 2\gamma - 2\log c) g(x) \> dx \\
	& \hskip 0.3in - \frac{2\pi}{c} \sum_{\ell = 1}^\infty \sigma(\ell) \e{\frac{-d\ell}{c}} \int_0^\infty Y_0 \pr{\frac{4\pi}{c} \sqrt{\ell x}} g(x) \> dx \\
	& \hskip 0.3in + \frac{4}{c} \sum_{\ell = 1}^{\infty} \sigma(\ell) \e{\frac {d\ell}{c}} \int_0^\infty K_0  \pr{\frac{4\pi}{c} \sqrt{\ell x}} g(x) \> dx. }
\end{lem}

Finally, we will eventually reduce our bound to applications of the following large sieve inequality involving $GL(1)$ harmonics as stated in Exercise 5, Chapter 7 in \cite{IK}.

\begin{lem} \label{lem:largesieve}
	For any complex numbers $\alpha_m, \beta_n$, we have
	\est{ \sum_{q \leq Q} \sumstar_{a \mod q} \left| \sumtwo_{\substack{m \leq M, n \leq N \\ (mn, q) = 1}} \alpha_m \beta_n \e{\frac{a m \bar n}{q}}\right|^2 \leq (Q^2 + MN) \| \alpha \|^2 \| \beta\|^2.}
\end{lem}

\section{First step toward the proof of Proposition \ref{prop:mainbound}}
By the asymptotic large sieve in Lemma \ref{lem:ALS}, we obtain that 
\est{ \frac{2}{\phi(q)}&\sum_{\substack{\chi \mod q \\ \chi(-1) = (-1)^k}} \sumh_{f \in \mathcal H_\chi} \left|\sum_{n \geq 1} \lambda_f(bcn) a_{bcn} \right|^2 \\
	& =   \frac{1}{q} \sum_{\substack{1 \leq t \leq T \\ (t, q) = 1}} \left( \frac{2\pi}{t}\right)^2 \sum_{1 \leq h \leq H} |P_{ht}(\alpha)|^2 + O \pr{q^\epsilon \pr{\frac Y{q^2} + \sqrt{\frac{Y}{qH}} }} \| \alpha \|^2}
where $T = Y/q$ and $1 \leq H \leq T.$ We choose $H = Y/q$ so that the error term is small. We note that since $Y \leq X \leq q^{2 + \epsilon}$
$$\| \alpha \|^2 \leq  \sum_{Y < bcn \leq 2Y} \frac{\sigma_4^2(n)}{bcn} \ll \frac{1}{bc} \pr{\log \frac{Y}{bc}}^{16} \ll q^\epsilon.$$

Throughout the paper, $\epsilon$ denotes an arbitrary small positive constant that may vary from term to term, but $\epsilon_1$ will be a small fixed constant, which can be chosen later. 

We divide $h$ and $t$ into dyadic intervals, and it is enough to consider 
\est{  \frac{1}{qT_1^2} \sum_{\substack{t \sim T_1 \\ (t, q) = 1}} \sum_{h} \Psi\left( \frac h{H_1}\right) |P_{ht}(\alpha)|^2,}where we remind the reader that $\Psi_1$ denotes a compactly supported smooth function and where we write $a \sim A$ as shorthand for $A < a \leq 2A.$ 
We know that $\sigma_4(n) = \sum_{n_1n_2 = n} \sigma(n_1) \sigma(n_2)$, and we can divide the sum over $n_1$ into dyadic intervals. Therefore we write 
\es{\label{def:Pht} P_{ht}(\alpha) &=  \sum_{n} \frac{\sigma_4(n)}{\sqrt{bcn}} \Psi_1\pr{\frac{bcn}{Y}} S(h\bar q, bcn; t) J_{k-1} \left( \frac{4\pi}{t} \sqrt{\frac{hbcn}{q}}\right) \\
	&= \sumd_{N} \sumstar_{r \mod t} \e{\frac{h\bar q \bar r}{t}} \sum_{n_1} \frac{\sigma_2(n_1)}{\sqrt{bcn_1}} \Psi\left( \frac {n_1}{N}\right)\mathcal N(n_1; r, t, h, bc), }
where 
\est{\mathcal N(n_1; r, t, h, bc) := \sum_{n_2} \frac{\sigma_2(n_2)}{\sqrt{n_2}}  J_{k-1} \left( \frac{4\pi}{t} \sqrt{\frac{hbcn_1n_2}{q}}\right) \Psi_1\pr{\frac{bcn_1n_2}{Y}}\e{\frac{ r bcn_1n_2 }{t}},}
and $\sumstar_{r \mod t}$ denotes a sum over $ 1 \leq r \leq t$ such that $(r, t) = 1$.  Since $n_1n_2 \ll Y$, by symmetry, we may without loss of generality assume that 

$$N_1 \ll \sqrt {\frac{Y}{bc}}.$$

 Thus it is enough to consider 
\est{\mathcal F(T_1, H_1; q) := \frac{1}{qT_1^2} \sum_{\substack{t \sim T_1 \\ (t, q) = 1}} \sum_{h} \Psi\left( \frac h{H_1}\right) \left|\sumstar_{r \mod t} \e{\frac{h\bar q \bar r}{t}} \sum_{n_1} \frac{\sigma_2(n_1)}{\sqrt{bcn_1}} \Psi\left( \frac {n_1}{N}\right)\mathcal N(n_1; r, t, h, bc) \right|^2. }

Let 
$$ \frac{bcn_1}{t} = \frac{m}{\eta}, $$
where $(m, \eta) = 1.$  Applying Voronoi summation formula in Lemma \ref{lem:voronoidn} to $\mathcal N(n_1; r, t, h, bc)$, we obtain that 
$$\mathcal N(n_1; r, t, h, bc) = R_1 + R_2 + R_3,$$
where 
\es{\label{def:R1} R_1 := \frac{1}{\eta} \int_0^\infty (\log x + 2\gamma - 2\log \eta) \frac{1}{\sqrt x} J_{k-1}\pr{\frac{4\pi}{t} \sqrt{\frac{hbcn_1x}{q}}} \Psi_1\pr{\frac{bcn_1x}{Y}}  \> dx; }
\es{\label{def:R2} R_2 := - \frac{2\pi}{\eta} \sum_{\ell = 1}^\infty \sigma_2(\ell) \e{\frac{-\bar m_\eta \bar r_\eta \ell}{\eta}} \int_0^\infty Y_0 \pr{\frac{4\pi}{\eta} \sqrt{\ell x}} \frac{1}{\sqrt x} J_{k-1}\pr{\frac{4\pi}{t} \sqrt{\frac{hbcn_1x}{q}}} \Psi_1\pr{\frac{bcn_1x}{Y}}  \> dx ;}

\es{\label{def:R3} R_3 := \frac{4}{\eta} \sum_{\ell = 1}^{\infty} \sigma_2(\ell) \e{\frac {\bar m_\eta \bar r_\eta \ell}{\eta}} \int_0^\infty K_0  \pr{\frac{4\pi}{\eta} \sqrt{\ell x}} \frac{1}{\sqrt x} J_{k-1}\pr{\frac{4\pi}{t} \sqrt{\frac{hbcn_1x}{q}}} \Psi_1\pr{\frac{bcn_1x}{Y}}  \> dx.}
Note here that $m \bar m_\eta \equiv 1 \mod \eta$ and $r \bar r_\eta \equiv 1 \mod \eta.$ Let 
\es{\label{def:Fi}\mathcal F_i(T_1, H_1; q) = \frac{1}{qT_1^2} \sum_{\substack{t \sim T_1 \\ (t, q) = 1}} \sum_{h} \Psi\left( \frac h{H_1}\right) \left| \sumstar_{r \mod t} \e{\frac{h\bar q \bar r}{t}} \sum_{n_1 } \frac{\sigma_2(n_1)}{\sqrt{bcn_1}}\Psi\left( \frac {n_1}{N}\right) R_i \right|^2}
for $i = 1, 2, 3.$
Since $(|a| + |b| + |c|)^2 \leq 3 (|a|^2 + |b|^2 + |c|^2) $, we have that

\est{\mathcal F(T_1, H_1; q) \ll \mathcal F_1(T_1, H_1; q) + \mathcal F_2(T_1, H_1; q) + \mathcal F_3(T_1, H_1; q).}
Hence Proposition \ref{prop:mainbound} will follow from the following Lemma.

\begin{lemma} \label{lem:boundFi} Let $0 < T_1 \leq T$ and $0 < H_1 \leq H$, where $T = H = Y/q$. Then  for $i = 1, 2, 3$, 
	$$  \mathcal F_i(T_1, H_1 ; q) \ll q^\epsilon,$$
where the implied constant depends on $\epsilon.$	
\end{lemma}

\section{Bounding $\mathcal F_1(T_1, H_1; q)$}
Firstly, $R_1$ defined in \eqref{def:R1} does not depend on variable $r$. The sum over $r$ is then the Ramanujan's sum, which is 
$$\sumstar_{r \mod t} \e{\frac{h\bar q \bar r}{t}} = \mu\left( \frac{t}{(t, h)}\right) \frac{\phi(t)}{\phi \pr{\frac{t}{(t, h)}}} \ll (t, h).$$
We change the variable from $x \rightarrow \frac{x}{bcn_1}$ inside the integral in \eqref{def:R1} and use Equation (\ref{asympJxbig}). Thus 

\es{ \label{eqn:boundR1} R_1 &= \frac{1}{\sqrt{bcn_1} \eta} \int_0^\infty (\log \frac{x}{bcn_1} + 2\gamma - 2\log \eta) \frac{1}{\sqrt x} J_{k-1}\pr{\frac{4\pi}{t} \sqrt{\frac{hx}{q}}} \Psi_1\pr{\frac{x}{Y}}  \> dx \\
	&\ll \frac{q^{\epsilon} (bcn_1, t)}{\sqrt{bcn_1 Y} T_1} \left( T_1 \sqrt{\frac{q}{H_1 Y}} \right)^{1/2} \left|  \int_0^\infty \e{  \pm \frac{2}{t} \sqrt{\frac{hx}{q}} }  \Psi_2 \pr{ \frac x Y} \> dx \right|, }
where $\Psi_2(x)$ is a smooth compactly supported function which may be expressed in terms of $W(x)$ and $\Psi_1(x)$. If $T_1 \ll \frac{1}{q^{\epsilon_1}}\sqrt{\frac{H_1Y}{q}}$, then we can integrate by parts many times and obtain that
$$ \mathcal F_1(T_1, H_1; q) \ll q^{-100},$$
which can be ignored. Otherwise, we bound the integral in Equation (\ref{eqn:boundR1}) trivially, use  $T_1 \gg \frac{1}{q^{\epsilon_1}}\sqrt{\frac{H_1Y}{q}}$  and derive that
\est{\mathcal F_1(T_1, H_1; q) &\ll  \frac{q^\epsilon}{qT_1^2}  \sum_{\substack{t \sim T_1 \\ (t, q) = 1 }} \sum_{h \sim H_1} \frac{Y (t, h)^2}{T_1^2} T_1 \sqrt{\frac{q}{H_1 Y}}   \left| \sum_{n_1 \sim N_1} \frac{(t, n_1) \sigma_2(n_1)}{n_1} \right|^2  \\
	&\ll q^{\epsilon} \frac{\sqrt{H_1 Y}}{T_1 \sqrt{q}}  \ll q^{\epsilon + \epsilon_1}.}
Upon choosing $\epsilon_1$ sufficiently small, we conclude the proof of Lemma \ref{lem:boundFi} for $i = 1.$

\section{Bounding $\mathcal F_2(T_1, H_1; q)$} \label{sec:F2} Since $m \bar m \equiv 1 \mod t$, $m \bar m \equiv 1 \mod \eta $ when $\eta | t$. Then after the change of variables $n_1 \rightarrow n$ and $x \rightarrow xY/(bcn)$, we write
\est{&  \sumstar_{r \mod t} \e{\frac{h\bar q \bar r}{t}} \sum_{n} \frac{\sigma_2(n)}{\sqrt{bcn}} R_2 \\
	&=  - 2\pi\frac{\sqrt{Y}}{bc}  \sumstar_{r  \mod t} \e{\frac{h\bar q \bar r}{t}} \sum_{n } \frac{\sigma_2(n)}{n \eta} \Psi\left( \frac {n}{N}\right)  \sum_{\ell = 1}^\infty \sigma_2(\ell) \e{\frac{-\bar m \bar r \ell}{\eta}} \\
	& \hskip 0.2in \times \int_0^\infty Y_0 \pr{\frac{4\pi}{\eta} \sqrt{\frac{\ell xY}{bcn}}}  J_{k-1}\pr{\frac{4\pi}{t} \sqrt{\frac{hxY}{q}}} \Psi_3\pr{x}  \> dx, }
where $\Psi_3(x) = \frac{\Psi_1(x)}{\sqrt x}$, and $R_2$ is defined in (\ref{def:R2}). Hence
\est{\mathcal F_2(T_1, H_1; q) &\ll  \frac{Y}{qT_1^2 (bc)^2} \sum_{\substack{t \sim T_1 \\ (t, q) = 1}} \sum_{h} \Psi\pr{\frac h{H_1}}  \bigg| \sumstar_{r  \mod t} \e{\frac{h\bar q \bar r}{t}} \sum_{n } \frac{\sigma_2(n)}{n \eta} \Psi\left( \frac {n}{N}\right)  \sum_{\ell = 1}^\infty \sigma_2(\ell) \e{\frac{-\bar m \bar r \ell}{\eta}} \\
	& \hskip 2in \times \int_0^\infty Y_0 \pr{\frac{4\pi}{\eta} \sqrt{\frac{\ell xY}{bcn}}}  J_{k-1}\pr{\frac{4\pi}{t} \sqrt{\frac{hxY}{q}}} \Psi_3\pr{x}  \> dx \bigg|^2   }
Recall that $\eta = \frac{t}{(bcn, t)}$.  Writing $d = (t, bc)$ we obtain that 
 
\est{\mathcal F_2(T_1, H_1; q) 
	&\ll  \frac{Y}{qT_1^2 (bc)^2} \sum_{\substack{d| bc \\ (d,q)=1}} \sum_{\substack{t \sim T_1/d \\ (t, q) = 1 \\ (t, bc/d)=1}}  \sum_{h} \Psi\pr{\frac h{H_1}}  \bigg| \sumstar_{r  \mod {td}} \e{\frac{h\bar q \bar r}{td}} \sum_{n } \frac{\sigma_2(n) (n, t)}{n t} \Psi\left( \frac {n}{N}\right)   \\
	& \hskip 0.2in \times \sum_{\ell = 1}^\infty \sigma_2(\ell) \e{\frac{-\overline {\frac{\frac{bc}{d}n}{(t, n)}} \bar r \ell}{\frac{t}{(t, n)}}}\int_0^\infty Y_0 \pr{\frac{4\pi (t, n)}{t} \sqrt{\frac{\ell xY}{bcn}}}  J_{k-1}\pr{\frac{4\pi}{td} \sqrt{\frac{hxY}{q}}} \Psi_3\pr{x}  \> dx \bigg|^2   }
Next, we remove the greatest common divisors $(t, n)$ and $(t, \ell)$ to facilitate an eventual application of the large sieve inequality as in Lemma \ref{lem:largesieve}. Thus
\es{\label{F2eqbeforecase} \mathcal F_2(T_1, H_1; q) 
	&\ll \frac{Y}{qT_1^4 (bc)^2} \sum_{\substack{d| bc \\ (d,q)=1}} d^2 \sum_{ \substack{g \leq T/d \\ (g, q) = 1}} \sum_{\substack{g_1 \leq \frac{T}{dg} \\ (g_1, q) = 1}} \sum_{\substack{t \sim \frac{T_1}{dgg_1} \\ (t, q) = 1 \\ (t, bc/d)=1}}  \sum_{h} \Psi\pr{\frac h{H_1}}   \bigg|  \sumstar_{r  \mod {tdgg_1}} \e{\frac{h\bar q \bar r}{tdgg_1}}    \\
	& \hskip 0.2in \times \sum_{ \substack{n \\ (n, t) = 1} } \frac{\sigma_2(ng)}{n} \Psi\left( \frac {ng}{N}\right) \sum_{\substack{(\ell, t) = 1}} \sigma_2(\ell g_1) \e{\frac{-\overline {\frac{bc}{d}n} \bar r \ell}{t}}\int_0^\infty Y_0 \pr{\frac{4\pi }{t} \sqrt{\frac{\ell xY}{bcngg_1}}} \\
	&\hskip 0.2in \times  J_{k-1}\pr{\frac{4\pi}{tdgg_1} \sqrt{\frac{hxY}{q}}} \Psi_3\pr{x}  \> dx \bigg|^2. }

We divide into 3 cases, depending on the size of $T_1$ with respect to $\sqrt{\frac{H_1Y}{q}}$, which are  $T_1 \gg  \sqrt{\frac{H_1Y}{q}}$,  $\frac{1}{q^{\epsilon_1}} \sqrt{\frac{H_1Y}{q}} \ll T_1 \ll  \sqrt{\frac{H_1Y}{q}}$, and $T_1 \ll \frac{1}{q^{\epsilon_1}} \sqrt{\frac{H_1Y}{q}} .$

\subsection*{Case 1: $T_1 \gg  \sqrt{\frac{H_1Y}{q}}$}

For this case, we prove the following Lemma.

\begin{lem} \label{lem:F2T1big} Let $T_1 \gg  \sqrt{\frac{H_1Y}{q}}$. Then 
	$$ \mathcal F_2(T_1, H_1 ; q) \ll q^\epsilon,$$
	where the implied constant depends on $\epsilon.$
	
\end{lem}
We square the absolute value in Equation \eqref{F2eqbeforecase} and obtain that
\es{\label{F2maineq} \mathcal F_2(T_1, H_1; q) &\ll \frac{Y}{qT_1^4 (bc)^2} \sum_{\substack{d| bc \\ (d,q)=1}} d^2 \sum_{ \substack{g \leq T/d \\ (g, q) = 1}} \sum_{\substack{g_1 \leq \frac{T}{dg} \\ (g_1, q) = 1}} \sum_{\substack{t \sim \frac{T_1}{dgg_1} \\ (t, q\frac{bc}{d}) = 1 }}\sumstar_{r \mod {tdgg_1}} \sumstar_{r' \mod {tdgg_1}}  \\
	&\times \sumtwo_{\substack{n, n' \\ (nn', t) = 1} } \frac{\sigma_2(ng)}{n} \Psi\left( \frac {ng}{N}\right)  \frac{\sigma_2(n'g)}{n'} \Psi\left( \frac {n'g}{N}\right) \\
&\times	\sumtwo_{\substack{\ell, \ell' \\ (\ell \ell', t) = 1}}\sigma_2(\ell g_1) \e{\frac{-\overline {\frac{bc}{d}n} \bar r \ell}{t}} \sigma_2(\ell' g_1) \e{\frac{-\overline {\frac{bc}{d}n} \bar r' \ell'}{t}} \int_0^\infty \int_0^\infty Y_0 \pr{\frac{4\pi }{t} \sqrt{\frac{\ell xY}{bcngg_1}}} \\
&\times Y_0 \pr{\frac{4\pi }{t} \sqrt{\frac{\ell' x'Y}{bcn'gg_1}}}   \Psi_3\pr{x} \Psi_3\pr{x'}  \mathcal S(H_1, x, x', t) \>dx \> dx', }
where 
\est{\mathcal S(H_1, x, x', t) := \sum_{h} \Psi\left( \frac h{H_1}\right)  J_{k-1}\pr{\frac{4\pi}{tdgg_1} \sqrt{\frac{hxY}{q}}} J_{k-1}\pr{\frac{4\pi}{tdgg_1} \sqrt{\frac{hx'Y}{q}}}  \e{\frac{h\bar q (\bar r - \bar {r}')}{tdgg_1}} . }

 Since $\frac{1}{T_1} \sqrt{\frac{H_1Y}{q}} \ll 1,$ we can treat the Bessel function $J_{k-1}(x)$ as a smooth function.  Using Equation (\ref{asympJxSm}) for $J_{k-1}(x)$, we obtain that $\mathcal S(H_1, x, x', t)$ is

\est{\sum_{\alpha = 0}^\infty  \frac{(-1)^\alpha}{\alpha ! (\alpha + k - 1)!} \pr{\frac{4\pi}{tdgg_1} \sqrt{\frac{xYH_1}{q}} }^{2\alpha + k - 1} \pr{\frac{4\pi}{tdgg_1} \sqrt{\frac{x'YH_1}{q}} }^{2\alpha + k - 1} \sum_{h} \mathscr F_\alpha\pr{\frac h {H_1}} \e{\frac{h\bar q (\bar r - \bar r')}{tdgg_1}}, }
where 
$$ \mathscr F_\alpha\pr{\frac h {H_1}}  := \Psi\left( \frac{h}{H_1}\right) \pr{\frac{h}{H_1}}^{2\alpha + k -1}. $$
Note that $\mathscr F_\alpha (x) $ is compactly supported and smooth.  Applying Poisson summation to the sum over $h$, we have that 

\est{&\sum_{h} \mathscr F_\alpha\pr{\frac h {H_1}} \e{\frac{h\bar q (\bar r - \bar r')}{tdgg_1}}  
	= H_1 \sum_{\beta \equiv - \bar q(\bar r - \bar r') \mod {tdgg_1}} \int_{\mathbb R} \mathscr F_\alpha \left( y \right) \e{-\frac{\beta H_1 y}{tdgg_1}} \> dy. }

When $|\beta| \gg \frac{T_1}{H_1} q^\epsilon$, we do integration by parts many times to see that the contribution from these terms is negligible. Hence we focus only on those terms for which $|\beta | \ll \frac{T_1}{H_1}q^\epsilon$. 
 Thus bounding the terms inside the sum over $\beta$ trivially, we have
\es{\label{Fseq1} \mathcal F_2&(T_1, H_1; q) \ll \frac{YH_1q^\epsilon}{qT_1^4 (bc)^2} \sum_{\substack{d| bc \\ (d,q)=1}} d^2 \sum_{ \substack{g \leq T/d \\ (g, q) = 1}} \sum_{\substack{g_1 \leq \frac{T}{dg} \\ (g_1, q) = 1}}  \sum_{\beta \ll \frac{T_1}{H_1} q^\epsilon} \sum_{\alpha = 0}^\infty  \frac{1}{\alpha ! (\alpha + k - 1)!} \left(\frac{1}{T_1} \sqrt{\frac{YH_1}{q}} \right)^{4\alpha + 2k - 2} \\
	&\times  \sum_{\substack{t \sim \frac{T_1}{dgg_1} \\ (t, q\frac{bc}{d}) = 1 }}  \sumstar_{\substack{r \mod {tdgg_1} \\ (\bar r + \beta q, tdgg_1) = 1}}   \bigg| \sumtwo_{\substack{n, n' \\ (nn', t) = 1} } \frac{\sigma_2(ng)}{n} \Psi\left( \frac {ng}{N}\right)  \frac{\sigma_2(n'g)}{n'} \Psi\left( \frac {n'g}{N}\right)\\
	&\times \sumtwo_{\substack{\ell, \ell' \\ (\ell \ell', t) = 1}}\sigma_2(\ell g_1) \e{\frac{-\overline {\frac{bc}{d}n} \bar r \ell}{t}} \sigma_2(\ell' g_1) \e{\frac{-\overline {\frac{bc}{d}n} (\bar r + \beta q) \ell'}{t}}\\
	& \times \int_0^\infty \int_0^\infty Y_0 \pr{\frac{4\pi }{t} \sqrt{\frac{\ell xY}{bcngg_1}}}  Y_0 \pr{\frac{4\pi }{t} \sqrt{\frac{\ell' x'Y}{bcn'gg_1}}}    \Psi_4\pr{x} \Psi_4\pr{x'}  \> dx \> dx' \bigg| \\
	&\ll  \frac{YH_1q^\epsilon}{qT_1^4 (bc)^2} \sum_{\substack{d| bc \\ (d,q)=1}} d^2 \sum_{ \substack{g \leq T/d \\ (g, q) = 1}} \sum_{\substack{g_1 \leq \frac{T}{dg} \\ (g_1, q) = 1}}  \sum_{\beta \ll \frac{T_1}{H_1} q^\epsilon} \sum_{\alpha = 0}^\infty  \frac{1}{\alpha ! (\alpha + k - 1)!} \left(\frac{1}{T_1} \sqrt{\frac{YH_1}{q}} \right)^{4\alpha + 2k - 2} \\
	& \times \sum_{\substack{t \sim \frac{T_1}{dgg_1} \\ (t, q\frac{bc}{d}) = 1 }}  \sumstar_{\substack{r \mod {tdgg_1} \\ (\bar r + \beta q, tdgg_1) = 1}} (\mathcal A + \mathcal A'),} 
where $\Psi_4(x) := \Psi_4(x, \alpha) = {x^{\alpha + k/2 - 1}} \Psi_3(x)$ is also compactly supported and smooth and moreover,
\es{ \label{Fseq2}\mathcal A := &\bigg| \sum_{\substack{n \\ (n, t) = 1} } \frac{\sigma_2(ng)}{n} \Psi\left( \frac {ng}{N}\right)  \sum_{\substack{\ell \\ (\ell , t) = 1}}\sigma_2(\ell g_1) \e{\frac{-\overline {\frac{bc}{d}n} \bar r \ell}{t}}  \int_0^\infty Y_0 \pr{\frac{4\pi }{t} \sqrt{\frac{\ell xY}{bcngg_1}}}   \Psi_4\pr{x}  \> dx \bigg|^2 , }
and 
\est{ \mathcal A' := & \bigg| \sum_{\substack{n \\ (n, t) = 1} } \frac{\sigma_2(ng)}{n} \Psi\left( \frac {ng}{N}\right)  \sum_{\substack{\ell \\ (\ell , t) = 1}}\sigma_2(\ell g_1) \e{\frac{-\overline {\frac{bc}{d}n} (\bar r + \beta q) \ell}{t}}  \int_0^\infty Y_0 \pr{\frac{4\pi }{t} \sqrt{\frac{\ell xY}{bcngg_1}}}  \Psi_4\pr{x}  \> dx \bigg|^2.}
By a change of variables, we see that both 
$$\sum_{\substack{t \sim \frac{T_1}{dgg_1} \\ (t, q\frac{bc}{d}) = 1 }}  \sumstar_{\substack{r \mod {tdgg_1} \\ (\bar r + \beta q, tdgg_1) = 1}} \mathcal A$$ and $$\sum_{\substack{t \sim \frac{T_1}{dgg_1} \\ (t, q\frac{bc}{d}) = 1 }}  \sumstar_{\substack{r \mod {tdgg_1} \\ (\bar r + \beta q, tdgg_1) = 1}} \mathcal A'$$ 
are bounded by 
\est{ dgg_1\sum_{\substack{t \sim \frac{T_1}{dgg_1}  }}  \sumstar_{\substack{r \mod {t} }} & \bigg| \sum_{\substack{n \\ (n, t) = 1} } \frac{\sigma_2(ng)}{n} \Psi\left( \frac {ng}{N}\right)  \sum_{\substack{\ell \\ (\ell , t) = 1}}\sigma_2(\ell g_1) \e{\frac{\bar {n}  r \ell}{t}} \\
	& \int_0^\infty Y_0 \pr{\frac{4\pi }{t} \sqrt{\frac{\ell xY}{bcngg_1}}}   \Psi_4\pr{x}  \> dx \bigg|^2,}
independently of $\beta$.

Therefore it is enough to consider 
\es{\label{Fseq3} \mathcal G_b(T_1, H_1; q) &:= \frac{Yq^\epsilon}{qT_1^3 (bc)^2} \sum_{\substack{d| bc \\ (d,q)=1}} d^3 \sum_{ \substack{g \leq T/d \\ (g, q) = 1}} g \sum_{\substack{g_1 \leq \frac{T}{dg} \\ (g_1, q) = 1}} g_1  \sum_{\alpha = 0}^\infty  \frac{1}{\alpha ! (\alpha + k - 1)!} \left(\frac{1}{T_1} \sqrt{\frac{YH_1}{q}} \right)^{4\alpha + 2k - 2} \\
	& \times \sum_{\substack{t \sim \frac{T_1}{dgg_1}  }}  \sumstar_{\substack{r \mod {t} }} \bigg| \sum_{\substack{n \\ (n, t) = 1} } \frac{\sigma_2(ng)}{n} \Psi\left( \frac {ng}{N}\right)  \sum_{\substack{\ell \\ (\ell , t) = 1}}\sigma_2(\ell g_1) \e{\frac{\bar {n} r \ell}{t}} \\
	&\hskip 2in \times \int_0^\infty Y_0 \pr{\frac{4\pi }{t} \sqrt{\frac{\ell xY}{bcngg_1}}}   \Psi_4\pr{x}  \> dx \bigg|^2 }
and show that 
$$ \mathcal G_b(T_1, H_1; q) \ll q^\epsilon.$$
To prove this bound, we consider 3 cases depending on the range of $\ell$ and write 
$$ \mathcal G_b(T_1, H_1; q) \ll \mathcal G_{b,1}(T_1, H_1; q) + \mathcal G_{b, 2}(T_1, H_1; q), $$
where $\mathcal G_{b, i}(T_1, H_1; q)$ has the same expression as $\mathcal G_b(T_1, H_1; q)$ but restricting the size of $\ell$ to $\ell \gg q^{\epsilon_1} \frac{T_1^2bcN}{d^2g^2g_1Y}$ and $\ell \ll q^{\epsilon_1} \frac{T_1^2bcN}{d^2g^2g_1Y}$ respectively.

\subsubsection*{Case 1.1: $\ell \gg q^{\epsilon_1} \frac{T_1^2bcN}{d^2g^2g_1Y}$} In this case, we use the expression for $Y_0(x)$ in Equation (\ref{asymY0}) and consider
\est{\sum_{\substack{n \\ \pr{n, t} = 1}} \frac{\sigma_2(ng) }{n } \Psi\left( \frac {ng}{N}\right) \sum_{\substack{\ell \gg q^{\epsilon_1} \frac{T_1^2bcN}{d^2g^2g_1Y} \\ (\ell, t) = 1}} \sigma_2(\ell g_1) \e{\frac{\bar {n}  r \ell}{t}}   \left( t \sqrt{\frac{bcngg_1}{\ell  Y}}\right)^{1/2} \int_0^\infty \e{ \pm\frac{2}{t} \sqrt{\frac{\ell  xY}{bcngg_1}}}  \frac{\Psi_4\pr{x}}{x^{1/4}}  \> dx. }
 
 We can then integrate by parts many times with respect to $x$ and obtain that 
 $ \mathcal G_{b,1}(T_1, H_1; q)  \ll q^{-100}.$
 
\subsubsection*{Case 1.2: $\ell \ll q^{\epsilon_1} \frac{T_1^2bcN}{d^2g^2g_1Y}$}
 Using the integral expression for $Y_0$ in Equation (\ref{Y0int}), we have for $0 < \sigma < 1$
\es{\label{eqn:Y0writeint} Y_0 \pr{\frac{4\pi}{t} \sqrt{\frac{\ell  x Y}{bcngg_1}}} = \int_{(-\sigma)} \gamma(s) \pr{\frac{4\pi}{t} \sqrt{\frac{\ell x Y}{bcngg_1}}}^{-s} \> ds \ + \frac{2}{\pi}\ln\pr{\frac{4\pi}{t} \sqrt{\frac{ x Y}{bcgg_1}}} + \frac{1}{\pi} \log \ell - \frac 1\pi \log n  + \kappa.}
For brevity, we deal only with the integral term, the other terms being slightly easier. The sum over $t$ from this integral of $\mathcal G_{b, 2}(T_1, H_1; q)$ 
can be written as
\est{& \sum_{\substack{t \sim \frac{T_1}{dgg_1}  \\ }}  \sumstar_{\substack{r \mod {t} }} \bigg| \int_0^\infty \Psi_3\pr{x} \int_{(-\sigma)} \gamma(s)  \pr{\frac{tdgg_1}{4\pi T_1\sqrt x}  }^s \sum_{\substack{n \\ \pr{n, t} = 1}} \frac{\sigma_2(ng)  }{n  } \pr{\frac{n}{N/g}}^{s/2} \Psi\left( \frac {ng}{N_1}\right)\\
	& \hskip 1in \times  \sum_{\substack{\ell \ll q^{\epsilon_1} \frac{T_1^2bcN}{d^2g^2g_1Y} \\ \pr{\ell, t} = 1}} \sigma_2(\ell g_1) \pr{\frac{\tfrac{T_1^2bcN}{d^2g^2g_1Y}}{\ell}}^{s/2} \e{\frac{\bar {n}  r \ell}{t}}       \> ds \> dx \bigg|^2 .}
By Cauchy-Schwarz's inequality and the large sieve inequality in Lemma \ref{lem:largesieve}, we obtain that the above is bounded by
\est{
	&\ll \int_0^\infty |\Psi_3\pr{x}| \int_{(-\sigma)} |\gamma(s)| \sum_{\substack{t \sim \frac{T_1}{dgg_1}  }}  \sumstar_{\substack{r \mod {t} }} \bigg| \sum_{\substack{n \\ \pr{n, t} = 1}} \frac{\sigma_2(ng)  }{n  } \pr{\frac{n}{N/g}}^{s/2} \Psi\left( \frac {ng}{N}\right) \\
	&  \hskip 1in \times  \sum_{\substack{ \ell \ll q^{\epsilon_1} \frac{T_1^2bcN}{d^2g^2g_1Y} \\ \pr{\ell, t} = 1}} \sigma_2(\ell g_1) \pr{\frac{\tfrac{T_1^2bcN}{d^2g^2g_1Y}}{\ell}}^{s/2} \e{\frac{\bar {n}  r \ell}{t}}       \bigg|^2  \> |ds| \> dx \\
	&\ll q^\epsilon \pr{ \frac{T_1^2}{d^2g^2g_1^2} + \frac{N}{g} \frac{T_1^2bcN}{d^2g^2g_1Y} } \pr{ \frac{g}{N}} \pr{\frac{T_1^2bcN}{d^2g^2g_1Y}  } \ll  q^\epsilon \frac{T_1^4bc}{d^4g^3g_1^2Y} .}
Therefore, 
\est{ 	\mathcal G_{b, 2}(T_1, H_1; q) \ll \frac{Yq^\epsilon}{qT_1^3 (bc)^2} \sum_{d| bc} d^3 \sum_{g \leq \frac{T_1}{d}} g  \sum_{g_1 \leq \frac{T_1}{dg}} g_1 \frac{T_1^4bc}{d^4g^3g_1^2Y}   \ll q^\epsilon . }

From Case 1.1 - 1.2, we deduce that
$ \mathcal G_b(T_1, H_1 ; q) \ll q^\epsilon$
and conclude that 
$$ \mathcal F_2(T_1, H_1; q) \ll q^\epsilon$$
as desired.

\subsection*{Case 2: $\frac{1}{q^{\epsilon_1}} \sqrt{\frac{H_1Y}{q}} \ll T_1 \ll  \sqrt{\frac{H_1Y}{q}}$}

For this case, we will show that
\begin{lem} \label{lem:F2T1medium} Let $\frac{1}{q^{\epsilon_1}} \sqrt{\frac{H_1Y}{q}} \ll T_1 \ll  \sqrt{\frac{H_1Y}{q}}$. Then 
	$$ \mathcal F_2(T_1, H_1 ; q) \ll q^\epsilon,$$
	where the implied constant depends on $\epsilon.$
	
\end{lem}


We first consider the sum over $h$ in Equation \eqref{F2maineq}. Using Equation \eqref{asympJxbig} to express $J_{k-1}(x)$, we have
\est{\mathcal S(H_1, x, x', t) = \mathcal S_1 (H_1, x, x', t) + \mathcal S_2 (H_1, x, x', t)  + \mathcal S_3 (H_1, x, x', t)  + \mathcal S_4 (H_1, x, x', t) ,}
where each $\mathcal S_i(H_1, x, x', t)$ is of the form 
\est{   c_k tdgg_1 \sqrt{\frac{q}{H_1Y\sqrt{x x'}}} \sum_{h} \Psi_5\left( \frac h{H_1}\right) \e{ \pm \frac{2}{tdgg_1} \sqrt{\frac{hxY}{q}} \pm \frac{2}{tdgg_1} \sqrt{\frac{hx'Y}{q}} }    \e{\frac{h\bar q (\bar r - \bar {r}')}{tdgg_1}},}
for some choice of the sign $\pm$, some constant $c_k$  and where $\Psi_5(x)$ is a product of $\Psi(x)/\sqrt x$ and $W$ or $\overline W$. In particular, $\Psi_5(x)$ is smooth and compactly supported. Similar to Case 1, we apply Poisson summation formula to the sum over $h$ and derive that 
\es{\label{eqn:possoinsumh} \mathcal S_i (H_1, x, x', t) &=  c_k tdgg_1 \sqrt{\frac{q}{H_1Y\sqrt{x x'}}}  H_1 \sum_{\beta \equiv - \bar q (\bar r - \bar r') \mod {tdgg_1}} \int_{\mathbb R} \Psi_5(z) \\
	&\ \ \ \ \ \times  \e{ \pm \frac{2}{tdgg_1} \sqrt{\frac{zx H_1Y}{q}} \pm \frac{2}{tdgg_1} \sqrt{\frac{zx' H_1Y}{q}} } \e{- \frac{\beta H_1 z}{tdgg_1}} \> dz. }
If $|\beta| \gg q^\epsilon \sqrt{\frac{Y}{H_1q}} $, then $|\beta| \gg q^{\epsilon} \frac {T_1}{H_1}$ since $ T_1 \ll \sqrt{\frac{H_1Y}{q}}.$ We do integration by parts many times 
and obtain that the contribution from these terms is $\ll q^{-100}.$ So we focus only terms from $|\beta| \ll q^\epsilon \sqrt{\frac{Y}{H_1q}},$ 
Similar to Case 1, the contribution from these terms  can be bounded by 


\es{\label{case2eq2} \mathcal G_m(T_1, H_1; q) &:= \frac{q^\epsilon YH_1}{qT_1^4(bc)^2} T_1 \sqrt{\frac q {H_1 Y}} \sqrt{\frac Y{H_1q}}\int_{\mathbb R} |\Psi_5(z) |  \sum_{d| bc} d^3 \sum_{g \leq \frac{T_1}{d}} g  \sum_{g_1 \leq \frac{T_1}{dg}} g_1 \sum_{\substack{t \sim \frac{T_1}{dgg_1}  }}   \sumstar_{\substack{r \mod t }} \\
	&\times  \bigg| \sum_{n } \frac{\sigma_2(ng)}{n} \Psi\left( \frac {ng}{N}\right) \sum_{\ell \geq 1} \sigma_2(\ell g_1) \e{\frac{\bar n  r \ell}{t}}  \int_0^\infty Y_0 \pr{\frac{4\pi}{t} \sqrt{\frac{\ell Yx}{bcngg_1}}} \\
	&\times \e{\pm \frac{2}{tdgg_1} \sqrt{\frac{zxYH_1}{q}}} \Psi_6\pr{x}  \> dx \bigg|^2  \> dz,}
where $\Psi_6(x) = \Psi_3(x)/\sqrt x.$ Note the factor in front of integral can be simplified to $\frac{q^\epsilon Y}{qT_1^3 (bc)^2}$. In order to show that $$ \mathcal G_{m}(T_1, H_1; q) \ll q^\epsilon,$$
 we separate into 2 cases depending on the range of $\ell$. We write 
\es{\label{Gm} \mathcal G_m(T_1, H_1; q) \ll \mathcal G_{m,1}(T_1, H_1; q) + \mathcal G_{m, 2}(T_1, H_1; q), }
where $\mathcal G_{m, i}(T_1, H_1; q)$ has the same expression as $\mathcal G_m(T_1, H_1; q)$ but restricting the size of $\ell$ to $\ell \gg q^{\epsilon_1} \frac{H_1Nbc}{qd^2g^2g_1}$ and $ \ell \ll q^{\epsilon_1} \frac{H_1Nbc}{qd^2g^2g_1}$.

\subsubsection*{Case 2.1}   $\ell \gg q^{\epsilon_1} \frac{H_1Nbc}{qd^2g^2g_1}. $ 
We use the expression for $Y_0(x)$ in Equation (\ref{asymY0}) and write the integral inside the absolute value of Equation \eqref{case2eq2} as
\es{\label{eqcase2.1} \pr{t \sqrt{\frac{bcngg_1}{\ell Y}}}^{1/2} \int_0^\infty \frac{1}{x^{1/4}}\e {\pm \frac{2}{t} \sqrt{\frac{\ell Yx}{bcngg_1}}} \e{\pm \frac{2}{tdgg_1} \sqrt{\frac{zxYH_1}{q}}} \Psi_6\pr{x}  \> dx.}

From the fact that $T_1 \ll  \sqrt{\frac{H_1Y}{q}},$
$\frac{H_1Nbc}{qd^2g^2g_1} \gg \frac{T_1^2bcN}{d^2g^2g_1 Y}, $ and so
$ \ell \gg q^{\epsilon_1} \frac{T_1^2bcN}{d^2g^2g_1 Y}. $ We then integrate by parts many times 
 and obtain that $\mathcal G_{m, 1} (T_1, H_1; q) \ll q^{-100}.$

\subsubsection*{Case 2.2: $\ell  \ll {q^{\epsilon_1}} \frac{NH_1bc}{qd^2g^2g_1}$}

Since $\frac{1}{q^{\epsilon_1}}\sqrt{\frac{H_1Y}{q}} \ll T_1 \ll \sqrt{\frac{H_1Y}{q}}$, then 
\es{ \label{eqn:ellrange1} \ell \ll q^{\epsilon_1} \frac{NH_1bc}{qd^2g^2g_1} \ll  q^{3\epsilon_1}\frac{T_1^2bcN}{d^2g^2g_1Y} .}

 We write $Y_0(x)$ using the integral representation in Equation \eqref{eqn:Y0writeint}. By the condition on $\ell$, we can treat $\mathcal G_{m, 2}(T_1, H_1; q)$ in the same manner as Case 1.2 and obtain that 
 $$ \mathcal G_{m, 2}(T_1, H_1; q) \ll q^{\epsilon}.$$
 
From Case 2.1-2.2 and Equation \eqref{Gm}, we conclude that $ \mathcal G_m(T_1, H_1; q) \ll q^{\epsilon}$
and this prove Lemma \ref{lem:F2T1medium}.

\subsection*{Case 3: $T_1 \ll \frac{1}{q^{\epsilon_1}} \sqrt{\frac{H_1Y}{q}}  $}
For this case, we show that
\begin{lem} \label{lem:F2T1small} Let $T_1 \ll \frac{1}{q^{\epsilon_1}} \sqrt{\frac{H_1Y}{q}}  $. Then 
	$$ \mathcal F_2(T_1, H_1 ; q) \ll q^\epsilon,$$
	where the implied constant depends on $\epsilon.$
	
\end{lem}
We now divide the range for $\ell$ into 2 ranges, which are  $ \frac 14 \frac{H_1Nbc}{qd^2g^2g_1} < \ell < 3 \frac{H_1Nbc}{qd^2g^2g_1}$ and the rest.

For the second range, we write the Bessel functions $J_{k-1}(x)$ and $Y_0(x)$ as in Equations (\ref{asympJxbig}) and (\ref{asymY0}), respectively. Then, the integral in Equation (\ref{F2eqbeforecase}) is of the form

\est{ \pr{t \sqrt{\frac{bcngg_1}{\ell Y}}}^{1/2} \pr{tdgg_1 \sqrt{\frac{q}{zYH_1}}}^{1/2} \int_0^\infty \frac{1}{x^{1/2}}\e {\pm \frac{2}{t} \sqrt{\frac{\ell Yx}{bcngg_1}}} \e{\pm \frac{2}{tdgg_1} \sqrt{\frac{zxYH_1}{q}}} \Psi_7\pr{x}  \> dx,}
where $\Psi_7(x)$ is a product of $\Psi_3(x)$ and $W$ or $\overline W$.  Since $T_1 \ll \frac{1}{q^{\epsilon_1}} \sqrt{\frac{H_1 Y}{q}}$, 
$$  \frac{2}{tdgg_1} \sqrt{\frac{zYH_1}{q x}}  \gg q^{\epsilon_1}.$$
If both terms inside the exponential function has the same sign then we can  integrate by parts many times and obtain that the contribution from these terms is negligible. However, if the signs are different, then  the size of the derivative of the phase inside the exponential is 
$$ \mathscr E(x) = \left| \frac{1}{t\sqrt x} \sqrt{\frac{\ell Y}{bcngg_1}} - \frac{1}{tdgg_1 \sqrt x} \sqrt{\frac{zYH_1}{q}} \right|. $$
If $\ell \geq 3 \frac{H_1Nbc}{qd^2g^2g_1}$, then 
$$ \mathscr E(x) =  \frac{1}{t\sqrt x} \sqrt{\frac{\ell Y}{bcngg_1}} - \frac{1}{tdgg_1 \sqrt x} \sqrt{\frac{zYH_1}{q}} \geq \pr{\sqrt{3} \sqrt{\frac Nn} - \sqrt 2}\frac{1}{tdgg_1 \sqrt x } \sqrt{\frac{H_1Y}{q}} \gg q^{\epsilon_1} . $$
Similarly, if $\ell \leq \frac{1}{4} \frac{H_1Nbc}{qd^2g^2g_1}$, then 
$$ \mathscr E(x) =    \frac{1}{tdgg_1 \sqrt x} \sqrt{\frac{zYH_1}{q}} - \frac{1}{t\sqrt x} \sqrt{\frac{\ell Y}{bcngg_1}} \geq  \pr{1 -  \frac 12 \sqrt{\frac Nn} }\frac{1}{tdgg_1  \sqrt x} \sqrt{\frac{H_1Y}{q}} \gg q^{\epsilon_1} . $$


 Note moreoever that $\mathscr E(x) \asymp \mathscr E^{(k)}(x)$ for any $k\ge 0$.  We then do integration by parts many times and derive that the contribution from these terms is bounded by $q^{-100}$.  

Now we can focus on terms from $ \frac{1}{4} \frac{H_1Nbc}{qd^2g^2g_1} < \ell  < 3 \frac{H_1Nbc}{qd^2g^2g_1} $. 
After squaring out Equation (\ref{F2eqbeforecase}), we obtain that the contribution from these terms has the same expression as Equation \eqref{F2maineq}, with the additional restriction $\frac{1}{4} \frac{H_1Nbc}{qd^2g^2g_1} < \ell, \ell'  < 3 \frac{H_1Nbc}{qd^2g^2g_1}$.
 
Next, we treat the sum over $h$ similarly to the beginning of Case 2 and obtain the same expression as in Equation  \eqref{eqn:possoinsumh}. Moreover, by the same argument as Case 2, the contribution from $|\beta| \gg q^{\epsilon} \sqrt{\frac Y{H_1q}}$ is negligible. Next, we use Equation (\ref{asymY0}) for $Y_0(x)$ when $|\beta| \ll q^{\epsilon} \sqrt{\frac{Y}{H_1q}}.$  Hence, to bound $\mathcal F_2(T_1, H_1; q)$, it is sufficient to bound


\est{& \mathcal G_s(T_1, H_1 ;q) := \frac{\sqrt {H_1}}{T_1^4 \sqrt{q} (bc)^{3/2}} \sum_{ |\beta| \ll q^{\epsilon}   \sqrt{\frac{Y}{H_1q}}}   \sum_{\substack{d| bc \\ (d,q)=1}} d^3 \sum_{ \substack{g \leq T/d \\ (g, q) = 1}} g^{3/2} \sum_{\substack{g_1 \leq \frac{T}{dg} \\ (g_1, q) = 1}} g_1^{3/2}\sum_{\substack{t \sim \frac{T_1}{dgg_1} \\ (t, q\frac{bc}{d}) = 1 }} t^2 \sumstar_{\substack{r \mod {tdgg_1} \\ (\bar r + \beta q, tdgg_1) = 1}}  \\
	&\times \sumtwo_{\substack{n, n' \\ (nn', t) = 1} } \frac{\sigma_2(ng)}{n^{3/4}} \Psi\left( \frac {ng}{N}\right)  \frac{\sigma_2(n'g)}{n'^{3/4}} \Psi\left( \frac {n'g}{N}\right) \sumtwo_{\substack{ \frac{1}{4} \frac{H_1Nbc}{qd^2g^2g_1} < \ell  < 3 \frac{H_1Nbc}{qd^2g^2g_1} \\ (\ell \ell', t) = 1}}\frac{\sigma_2(\ell g_1)}{\ell^{1/4}} \e{\frac{-\overline {\frac{bc}{d}n} \bar r \ell}{t}} \\
	& \times \frac{\sigma_2(\ell' g_1)}{\ell'^{1/4}} \e{\frac{-\overline {\frac{bc}{d}n'} (\bar r + \beta q) \ell'}{t}} \mathcal J(n, n', \ell, \ell'), }
where  
\est{\mathcal J(n, n', \ell, \ell') := &\int_{0}^\infty \Psi_5(z) \int_0^\infty \int_0^\infty  {\Psi_7\pr{x}}{\Psi_7\pr{x'}}  \e{\pm \frac{2 }{t} \sqrt{\frac{\ell xY}{bcngg_1}} \pm \frac{2 }{t} \sqrt{\frac{\ell' x'Y}{bcn'gg_1}}} \\
	&\times \e{ \pm \frac{2}{tdgg_1} \sqrt{\frac{zx H_1Y}{q}} \pm \frac{2}{tdgg_1} \sqrt{\frac{zx' H_1Y}{q}} } \e{- \frac{\beta H_1 z}{tdgg_1}}   \>dx \> dx' \> dz,}
and $\Psi_7(x)$ is the product between $\frac{\Psi_3(x)}{\sqrt x}$ and $W$ or $\overline W$. Due to the range of $T_1$ and $\ell, \ell'$, we have that
\est{ 
	\frac{2}{tdgg_1} \sqrt{\frac{zx H_1Y}{q}}, \ \frac{2}{tdgg_1} \sqrt{\frac{zx' H_1Y}{q}}\  \gg \ q^{\epsilon_1} }
and
\est{
	 \frac{2 }{t} \sqrt{\frac{\ell xY}{bcngg_1}}, \  \frac{2 }{t} \sqrt{\frac{\ell' x'Y}{bcn'gg_1}} \  \gg \  \frac{1}{tdgg_1} \sqrt{\frac{ H_1Y}{q}}  \ \gg \ q^{\epsilon_1}.}

If the signs in front of $ \frac{2}{tdgg_1} \sqrt{\frac{zx H_1Y}{q}}$ and $\frac{2 }{t} \sqrt{\frac{\ell xY}{bcngg_1}}$ in $\mathcal J(n, n', \ell, \ell')$ are the same, we can integrate by part many times with respect to $x$ and show that the contribution from these terms is negligible. The same is true for the signs in front of $\frac{2}{tdgg_1} \sqrt{\frac{zx' H_1Y}{q}}$ and $ \frac{2 }{t} \sqrt{\frac{\ell' x'Y}{bcn'gg_1}}$. This motivates us to define 
\est{\mathcal J_1(n, n', \ell, \ell') := &\int_{0}^\infty \Psi_5(z) \int_0^\infty \int_0^\infty  \e{ -\frac{2 }{t} \sqrt{\frac{\ell xY}{bcngg_1}} - \frac{2 }{t} \sqrt{\frac{\ell' x'Y}{bcn'gg_1}}} \\
	&\times \e{  \frac{2}{tdgg_1} \sqrt{\frac{zx H_1Y}{q}} + \frac{2}{tdgg_1} \sqrt{\frac{zx' H_1Y}{q}} } \e{- \frac{\beta H_1 z}{tdgg_1}}  {\Psi_7\pr{x}}{\Psi_7\pr{x'}}  \>dx \> dx' \> dz,}
and
\est{\mathcal J_2(n, n', \ell, \ell') := &\int_{0}^\infty \Psi_5(z) \int_0^\infty \int_0^\infty  \e{ -\frac{2 }{t} \sqrt{\frac{\ell xY}{bcngg_1}} + \frac{2 }{t} \sqrt{\frac{\ell' x'Y}{bcn'gg_1}}} \\
	&\times \e{  \frac{2}{tdgg_1} \sqrt{\frac{zx H_1Y}{q}} - \frac{2}{tdgg_1} \sqrt{\frac{zx' H_1Y}{q}} } \e{- \frac{\beta H_1 z}{tdgg_1}}  {\Psi_7\pr{x}}{\Psi_7\pr{x'}}  \>dx \> dx' \> dz.}

By symmetry it is enough to consider only $\mathcal J_1$ and $\mathcal J_2.$   Further, let $\mathcal G_{s, i}(T_1, H_1; q)$ be the same expression as $\mathcal G(T_1, H_1; q)$ but replacing $\mathcal J(T_1, H_1; q)$ by $\mathcal J_i(T_1, H_1; q)$ for $i = 1, 2.$ It now suffices to show that for $i=1, 2$,
$$ \mathcal G_{s, i}(T_1, H_1; q) \ll q^{\epsilon}. $$


\subsubsection*{ Calculation of $\mathcal G_{s, 1}(T_1, H_1; q)$} 
By the change of variables $ \sqrt x = u - \sqrt{x'}$, 
\es{\label{def:J1} \mathcal J_1(n, n', \ell, \ell') 
	&= \int_{0}^\infty \Psi_5(z) \int_0^\infty \int_0^\infty  \e{ -\frac{2 u}{t} \sqrt{\frac{\ell Y}{bcngg_1}}  +\frac{2 }{t} \sqrt{\frac{ x' Y}{bcgg_1}} \pr{ \sqrt{\frac{\ell}{n}} - \sqrt{\frac{\ell'}{n'}}} } \\
	&\times \e{  \frac{2u}{tdgg_1} \sqrt{\frac{z H_1Y}{q}} - \frac{\beta H_1 z}{tdgg_1}}  2(u - \sqrt{x'}){\Psi_7\pr{(u - \sqrt {x'})^2}}{\Psi_7\pr{x'}}  \>du \> dx' \> dz.}
Note that the integrand vanishes when $u$ is outside the range $1 \leq u - \sqrt{x'} \leq \sqrt 2$\ \  because of the support of $\Psi_7(x)$. We consider the derivative with respect to $z$ of the expression inside the exponential, which is 
$$ \mathcal D_\beta(z) =  \frac{u}{tdgg_1} \sqrt{\frac{ H_1Y}{qz}} - \frac{\beta H_1}{tdgg_1} .$$ 
If $|D_\beta(z)| \geq q^{\epsilon_1/2}$, we can do integration by parts many times and obtain that the contribution from these terms is negligible. Hence it suffices to consider $ u \in \mathcal I_{z, \beta}$, where $\mathcal I_{z, \beta}$ is the interval 
$$\pr{\sqrt z \beta \sqrt{\frac{qH_1}{Y}} - 2T_1\sqrt{\frac{q}{H_1Y}}\sqrt z q^{\epsilon_1/2} , \sqrt z \beta \sqrt{\frac{qH_1}{Y}} + 2T_1\sqrt{\frac{q}{H_1Y}}\sqrt z q^{\epsilon_1/2}}.$$
Note that the length of $\mathcal I_{z, \beta} $ is 
\es{\label{lengthI} |\mathcal I_{z, \beta}| \asymp T_1\sqrt{\frac{q}{H_1Y}} q^{\epsilon_1/2} . }

Let $R = \frac{T_1}{q^{\epsilon_1/4}} \sqrt{\frac{q}{H_1Y}}$. If $|\ell' n - \ell n'| \gg q^{\epsilon_1/2} \frac{N}{g} \frac{NH_1bc}{qd^2g^2g_1} R,$ then 
\begin{align*}
\left| \sqrt{\frac{l}{n}} - \sqrt{\frac{l'}{n'}} \right| 
&\asymp \left| \frac{l}{n} - \frac{l'}{n'} \right| \sqrt{\frac{N}{g} \frac{qd^2g^2g_1}{NH_1bc}} \\
&\gg q^{\epsilon_1/2} \frac{H_1bc}{qd^2gg_1} R\sqrt{\frac{N}{g} \frac{qd^2g^2g_1}{NH_1bc}}\\ 
&\gg q^{\epsilon_1/4} T_1 \sqrt{\frac{bc}{d^2gg_1Y}}\\
&\asymp q^{\epsilon_1/4} \frac{t}{\sqrt{Y}} \sqrt{bcgg_1},
\end{align*}for $t\sim \frac{T_1}{dgg_1}$.
Thus, integration by parts many times with respect to $x'$ shows that the contribution from these terms is negligible. So we now assume
\es{\label{eqn:thinregion} |\ell' n - \ell n'| \ll q^{\epsilon_1/2} \frac{N}{g} \frac{NH_1bc}{qd^2g^2g_1} R.}

In order to deal with the condition in \eqref{eqn:thinregion}, we divide the intervals for $\ell, \ell'$ into intervals of length $R\frac{NH_1bc}{qd^2g^2g_1} $ and $n, n'$ into intervals of length $R \frac{N}{g}.$  Trivially, we need $\asymp \frac{1}{R^4}$ such tuples $(\mathcal I_{\mathfrak n}, \mathcal I_{\mathfrak n'}, \mathcal I_{\mathfrak l}, \mathcal I_{\mathfrak l'})$ to cover the whole range.  However, for fixed  $\mathcal I_{\mathfrak n}, \mathcal I_{\mathfrak l}, \mathcal I_{\mathfrak l'}$, where $\mathfrak n, \mathfrak n', \mathfrak l, \mathfrak l'$ are the left-ended points of intervals $\mathcal I_{\mathfrak n}, \ \mathcal I_{\mathfrak n'},\ \mathcal I_{\mathfrak l},\ \mathcal I_{\mathfrak l'}$, respectively, we may assume by \eqref{eqn:thinregion} that 
\begin{equation}\label{eqn:intnumber}
\left| \frac{\ell'n}{\ell} - n'\right| \ll q^{\epsilon_1/2} \frac{N}{g} R,
\end{equation}   for some $n \in \mathcal I_{\mathfrak n}, n' \in \mathcal I_{\mathfrak n'}, l \in \mathcal I_{\mathfrak l}$ and $l' \in \mathcal I_{\mathfrak l'}$.  However, due to our restriction on the length of the intervals, this implies that in fact \eqref{eqn:intnumber} is also satisfied by $\mathfrak{n, n', l}$ and $\mathfrak l'$.  Thus, there are $O(q^{\epsilon_1/2})$ choices for $I_{\mathfrak n'}$. Hence there are only $O\pr{\frac{q^{\epsilon_1/2}}{R^3}}$ relevant four tuples $(\mathcal I_{\mathfrak n}, \mathcal I_{\mathfrak n'}, \mathcal I_{\mathfrak l}, \mathcal I_{\mathfrak l'})$ with endpoints satisfying \eqref{eqn:thinregion}, and we obtain that

\es{\label{eqn:Gs1_f} &\mathcal G_{s, 1}(T_1, H_1; q) \ll \frac{\sqrt {H_1}}{T_1^2 \sqrt{q} (bc)^{3/2}}\sum_{ 0 \leq \beta \ll q^{\epsilon}    \sqrt{\frac{Y}{H_1q}}} \int_0^{\infty} | \Psi_5(z)| \int_0^\infty \int_{\mathcal I_{z, \beta}} \left| (u - \sqrt{x'}){\Psi_7\pr{(u - \sqrt {x'})^2}}\right| \\
	& \hskip 1in \times   \sum_{\substack{d| bc \\ (d,q)=1}} d \sum_{ \substack{g \leq T/d \\ (g, q) = 1}} g^{-1/2} \sum_{\substack{g_1 \leq \frac{T}{dg} \\ (g_1, q) = 1}} g_1^{-1/2} \sumt_{(\mathcal I_{\mathfrak n}, \mathcal I_{\mathfrak n'}, \mathcal I_{\mathfrak l}, \mathcal I_{\mathfrak l'})} S(\mathcal I_{\mathfrak n}, \mathcal I_{\mathfrak n'}, \mathcal I_{\mathfrak l}, \mathcal I_{\mathfrak l'}) {\Psi_7\pr{x'}}\>du \> dx' \> dz,}
where $\sumt$ is the sum over the relevant tuples with all points satisfying \eqref{eqn:intnumber}, and
\est{
	S(\mathcal I_{\mathfrak n}, \mathcal I_{\mathfrak n'}, \mathcal I_{\mathfrak l}, \mathcal I_{\mathfrak l'}) &= \sum_{\substack{t \sim \frac{T_1}{dgg_1} \\ (t, q\frac{bc}{d}) = 1 }}  \sumstar_{\substack{r \mod {tdgg_1} \\ (\bar r + \beta q, tdgg_1) = 1}} \\
	&\times \left| \sumtwo_{\substack{n \in \mathcal I_{n}, \ n' \in \mathcal I_{n'} \\ (nn', t) = 1} } \frac{\sigma_2(ng)}{n^{3/4}} \Psi\left( \frac {ng}{N}\right)  \frac{\sigma_2(n'g)}{n'^{3/4}} \Psi\left( \frac {n'g}{N}\right) \right. \\
	&\times \sumtwo_{\substack{  \ell \in \mathcal I_{\ell}, \ \ell' \in \mathcal I_{\ell'} \\ (\ell \ell', t) = 1}}  \frac{\sigma_2(\ell g_1)}{\ell^{1/4}} \e{\frac{-\overline {\frac{bc}{d}n} \bar r \ell}{t}}  \frac{\sigma_2(\ell' g_1)}{\ell'^{1/4}} \e{\frac{-\overline {\frac{bc}{d}n'} (\bar r + \beta q) \ell'}{t}}    \\
	& \left. \times  \e{ -\frac{2 u}{t} \sqrt{\frac{\ell Y}{bcngg_1}}  +\frac{2 }{t} \sqrt{\frac{ x' Y}{bcgg_1}} \pr{ \sqrt{\frac{\ell}{n}} - \sqrt{\frac{\ell'}{n'}}} }\right|.   }
Further
$$ |S(\mathcal I_{\mathfrak n}, \mathcal I_{\mathfrak n'}, \mathcal I_{\mathfrak l}, \mathcal I_{\mathfrak l'})| \leq |S_1(\mathcal I_{\mathfrak n}, \mathcal I_{\mathfrak l})| + |S_2(\mathcal I_{\mathfrak n'}, \mathcal I_{\mathfrak l'})|,$$
where 
\est{ S_1 := &S_1(\mathcal I_{\mathfrak n}, \mathcal I_{\mathfrak l}) = \sum_{\substack{t \sim \frac{T_1}{dgg_1} \\ (t, q\frac{bc}{d}) = 1 }}  \sumstar_{\substack{r \mod {tdgg_1} \\ (\bar r + \beta q, tdgg_1) = 1}} \\
	&\times   \bigg| \sum_{\substack{n \in \mathcal I_{\mathfrak n} \\ (n, t) = 1} }  \frac{\sigma_2(ng)}{n^{3/4}} \Psi\left( \frac {ng}{N}\right)   \sum_{\substack{  \ell \in \mathcal I_{\mathfrak l} \\ (\ell, t) = 1}}  \frac{\sigma_2(\ell g_1)}{\ell^{1/4 }} \e{\frac{-\overline {\frac{bc}{d}n} \bar r \ell}{t}}  \e{ \pr{-\frac{2 u}{t} \sqrt{\frac{ Y}{bcgg_1}}  +\frac{2 }{t} \sqrt{\frac{ x' Y}{bcgg_1}} } \sqrt{\frac{\ell}{n}} } \bigg|^2, }
and
\est{ S_2 := &S_2(\mathcal I_{\mathfrak n'}, \mathcal I_{ \mathfrak l'}) =  \sum_{\substack{t \sim \frac{T_1}{dgg_1} \\ (t, q\frac{bc}{d}) = 1 }}  \sumstar_{\substack{r \mod {tdgg_1} \\ (\bar r + \beta q, tdgg_1) = 1}} \bigg| \sum_{ \substack{n' \in \mathcal I_{\mathfrak n'} \\ (n', t) = 1} } \frac{\sigma_2(n'g)}{n'^{3/4 }} \Psi\left( \frac {n'g}{N}\right) \\
	&\hskip 1in \times  \sum_{\substack{   \ell' \in \mathcal I_{\mathfrak l'} \\ (\ell', t) = 1}}    \frac{\sigma_2(\ell' g_1)}{\ell'^{1/4}} \e{\frac{-\overline {\frac{bc}{d}n'} (\bar r + \beta q) \ell'}{t}}   \e{ - \frac{2 }{t} \sqrt{\frac{ x'Y}{bcgg_1}} \sqrt{\frac{\ell'}{n'}} }\bigg|^2 . }
By a change of variables in $r$, we have

\es{ \label{def:S1}  S_1(\mathcal I_{\mathfrak n}, \mathcal I_{\mathfrak l})
	&\leq 	dgg_1 \sum_{\substack{t \sim \frac{T_1}{dgg_1}  }}  \sumstar_{\substack{r \mod {t} }}\bigg| \sum_{\substack{n \in \mathcal I_{\mathfrak n} \\ (n, t) = 1} } \frac{\sigma_2(ng)}{n^{3/4}} \Psi\left( \frac {ng}{N}\right)   \sum_{\substack{  \ell \in \mathcal I_{\mathfrak l} \\ (\ell, t) = 1}}  \frac{\sigma_2(\ell g_1)}{\ell^{1/4 }} \\
	&\hskip 1in \times \e{ \pr{-\frac{2 u}{t} \sqrt{\frac{ Y}{bcgg_1}}  +\frac{2 }{t} \sqrt{\frac{ x' Y}{bcgg_1}} } \sqrt{\frac{\ell}{n}} }  \e{\frac{ r \bar n  \ell}{t}} \bigg|^2 .}

Similarly, we have
 
\es{\label{def:S2} &S_{2}(\mathcal I_{\mathfrak n'}, \mathcal I_{\mathfrak l'}) 
	\leq 	dgg_1  \sum_{\substack{t \sim \frac{T_1}{dgg_1}  }}  \sumstar_{\substack{r \mod {t} }}  \bigg|   \sum_{ \substack{n' \in \mathcal I_{\mathfrak n'} \\ (n', t) = 1} } \frac{\sigma_2(n'g)}{n'^{3/4 }} \Psi\left( \frac {n'g}{N}\right)  \sum_{\substack{   \ell' \in \mathcal I_{\mathfrak l'} \\ (\ell', t) = 1}}    \frac{\sigma_2(\ell' g_1)}{\ell'^{1/4}}         \\
	&\hskip 3in \times  \e{ - \frac{2 }{t} \sqrt{\frac{ x'Y}{bcgg_1}} \sqrt{\frac{\ell'}{n'}} } \e{\frac{ r \bar n'  \ell'}{t}}\bigg|^2  .}
Hence
\est{
	\mathcal G_{s, 1}(T_1, H_1; q) \ll \mathcal P_1(T_1, H_1; q) + \mathcal P_2(T_1, H_1; q),}
where $\mathcal P_i(T_1, H_1; q)$ has the same expression as the right hand side of \eqref{eqn:Gs1_f}  but we replace $S(\mathcal I_{\mathfrak n}, \mathcal I_{\mathfrak n'}, \mathcal I_{\mathfrak l}, \mathcal I_{\mathfrak l'})$ by the right hand side of $S_i$ in \eqref{def:S1} and \eqref{def:S2}. Hence to show that $\mathcal G_{s, 1}(T_1, H_1; q) \ll q^{\epsilon}$, it is enough to prove the following.
\begin{lem} Let all notations be defined as above. Then for $i = 1, 2$, we have
	\est{ S_i \ll q^\epsilon   \pr{\frac{T_1^2}{dgg_1} + R^2 \frac{N}{g} \frac{NH_1bc}{qdg}} R^2 \sqrt{\frac{H_1bc}{qd^2gg_1}}.}
	From this we deduce that
	\est{P_i(T_1, H_1; q) \ll q^{\epsilon}.}
\end{lem}

\begin{proof} We start with bounding $S_2(\mathcal I_{\mathfrak n'}, \mathcal I_{\mathfrak l'}).$  We will eventually apply the large sieve Lemma \ref{lem:largesieve}; in order to do this, we wish to separate $l'$ from $n'$.
We recall that $ \mathfrak n, \mathfrak n', \mathfrak l, \mathfrak l'$ are the left-ended points of intervals $\mathcal I_{\mathfrak n}, \  \mathcal I_{\mathfrak n'}, \  \mathcal I_{\mathfrak l}, \ \mathcal I_{\mathfrak l'}   $. Let 
$$f(x) = \e{x}. $$ 
We will write a Taylor polynomial for $\e{ - \frac{2 }{t} \sqrt{\frac{ x'Y}{bcgg_1} }\sqrt{\frac{\ell'}{n'}} }$ around $x = - \frac{2 }{t} \sqrt{\frac{ x'Y}{bcgg_1} }\sqrt{\frac{\mathfrak l'}{\mathfrak n'} } $. Hence
\est{ \e{ - \frac{2 }{t} \sqrt{\frac{\ell' x'Y}{bcn'gg_1}} } =  \sum_{\alpha = 0}^{\infty}   \frac{1}{\alpha!}f^{(\alpha)}\pr{  - \frac{2 }{t} \sqrt{\frac{\mathfrak l' x'Y}{bc\mathfrak n'gg_1}}} \pr{ - \frac{2 }{t} \sqrt{\frac{ x'Y}{bcgg_1}}  \pr{ \sqrt{\frac{\ell'}{n'}} - \sqrt{\frac{\mathfrak l'}{\mathfrak n'}}  }  }^\alpha.}
Moreover since $|\ell' - \mathfrak l'| \ll R \frac{NH_1bc}{qd^2g^2g_1}$ and $|n' - \mathfrak n'| \ll R \frac{N}{g}$, we have
\es{\label{boundtaylor} &f^{(\alpha)}\pr{  - \frac{2 }{t} \sqrt{\frac{\mathfrak l' x'Y}{bc\mathfrak n'gg_1}}} \pr{ - \frac{2 }{t} \sqrt{\frac{ x'Y}{bcgg_1}}  \pr{ \sqrt{\frac{\ell'}{n'}} - \sqrt{\frac{\mathfrak l'}{\mathfrak n'}}  }  }^\alpha \\
	&\ll_{\alpha} \pr{\frac{1}{t} \sqrt{\frac{ Y}{bcgg_1}}}^\alpha \sum_{j = 0}^\alpha {\alpha \choose j} \left|  \frac{\sqrt {\ell'} - \sqrt{\mathfrak l'}}{\sqrt{n'}}  \right|^j  \left|  \sqrt{\mathfrak l'} \pr{\frac{1}{\sqrt{n'}} - \frac{1}{\sqrt{\mathfrak n'}}}  \right|^{\alpha - j}  \\
	&\ll_{\alpha} \pr{\frac{1}{t} \sqrt{\frac{ Y}{bcgg_1}}}^\alpha \sum_{j = 0}^\alpha {\alpha \choose j} \left|  \frac{\ell - \mathfrak l'}{\sqrt{n'}(\sqrt {\ell'} + \sqrt{\mathfrak l'})}  \right|^j  \left|  \sqrt{\mathfrak l'} \frac{\mathfrak n' - n'}{\sqrt{n'\mathfrak n'} ( \sqrt{n'} + \sqrt{\mathfrak n'})}   \right|^{\alpha - j} \\
	&\ll_{\alpha} R^\alpha \left(  \frac{1}{t} \sqrt{\frac{ Y}{bcgg_1}} \sqrt{ \frac{NH_1bc}{qd^2g^2g_1}} \sqrt{\frac{1}{N/g}}\right)^{\alpha} \ll_{\alpha} \pr{ \frac{R}{T_1} \sqrt{\frac{H_1Y}{q}}}^{\alpha}  \ll q^{-\alpha \epsilon_1/4}. }
Choosing $B$ such that $q^{-B\epsilon_1/4} \ll q^{-100},$ we obtain that
\est{ f\pr{ - \frac{2 }{t} \sqrt{\frac{\ell' x'Y}{bcn'gg_1}} } =  \sum_{0 \leq \alpha \leq B }   \frac{1}{\alpha!} f^{(\alpha)}\pr{  - \frac{2 }{t} \sqrt{\frac{\mathfrak l' x'Y}{bc\mathfrak n'gg_1}}} \pr{ - \frac{2 }{t} \sqrt{\frac{ x'Y}{bcgg_1}}  \pr{ \sqrt{\frac{\ell'}{n'}} - \sqrt{\frac{\mathfrak l'}{\mathfrak n'}}  }  }^\alpha + O(q^{-100}).}
Therefore, 
\est{ &S_{2}(\mathcal I_{\mathfrak n'}, \mathcal I_{\mathfrak l'}) 
	\ll 	dgg_1  \sum_{0 \leq \alpha \leq B}  \pr{\frac{dgg_1}{T_1} \sqrt{\frac{ Y}{bcgg_1}}}^{2\alpha}   \sum_{ 0 \leq j \leq \alpha} \sum_{\substack{t \sim \frac{T_1}{dgg_1}  }}  \sumstar_{\substack{r \mod {t} }}  \bigg|   \sum_{ \substack{n' \in \mathcal I_{\mathfrak n'} \\ (n', t) = 1} } \frac{\sigma_2(n'g)}{n'^{3/4 } (\sqrt{n'})^j}   \pr{\frac{1}{\sqrt{n'}} - \frac{1}{\sqrt{\mathfrak n'}}  }^{\alpha - j}   \\
	&\hskip 2in \times  \Psi\left( \frac {n'g}{N}\right)      \sum_{\substack{   \ell' \in \mathcal I_{\mathfrak l'} \\ (\ell', t) = 1}}    \frac{\sigma_2(\ell' g_1)}{\ell'^{1/4}} (\sqrt{\mathfrak l'})^{\alpha-j}     \pr{  {\sqrt {\ell'} - \sqrt{\mathfrak l'}}  }^j  \e{\frac{ r \bar n'  \ell'}{t}}\bigg|^2  + O(q^{-50}) .}
Then by large sieve inequality in Lemma \ref{lem:largesieve} and the same computation as Equation \eqref{boundtaylor}, we have
\est{
	&S_{2}(\mathcal I_{\mathfrak n'}, \mathcal I_{\mathfrak l'}) 
	\ll 	dgg_1 q^\epsilon \sum_{0 \leq \alpha \leq B}  \pr{\frac{dgg_1}{T_1} \sqrt{\frac{ Y}{bcgg_1}}}^{2\alpha}  \pr{\frac{T_1^2}{d^2g^2g_1^2} + R^2 \frac{N}{g} \frac{NH_1bc}{qd^2g^2g_1}}  \\
	&\hskip 1in \times \sum_{n' \in \mathcal I_{\mathfrak n'}} \frac{1}{n'^{3/2 } (\sqrt{n'})^{2j}}   \left|\frac{1}{\sqrt{n'}} - \frac{1}{\sqrt{\mathfrak n'}}  \right|^{2\alpha - 2j}   \sum_{\substack{   \ell' \in \mathcal I_{\mathfrak l'} }}    \frac{1}{\ell'^{1/2}} (\sqrt{\mathfrak l'})^{2\alpha-2j}     \left|  {\sqrt {\ell'} - \sqrt{\mathfrak l'}}  \right|^{2j}  \\
	&\ll  	dgg_1 q^\epsilon \sum_{0 \leq \alpha \leq B}  \pr{R\frac{dgg_1}{T_1} \sqrt{\frac{ Y}{bcgg_1}} \sqrt{ \frac{NH_1bc}{qd^2g^2g_1}} \sqrt{\frac{1}{N/g}}}^{2\alpha}  \pr{\frac{T_1^2}{d^2g^2g_1^2} + R^2 \frac{N}{g} \frac{NH_1bc}{qd^2g^2g_1}} R^2 \sqrt{\frac{H_1bc}{qd^2gg_1}} \\
  &\ll   q^\epsilon   \pr{\frac{T_1^2}{dgg_1} + R^2 \frac{N}{g} \frac{NH_1bc}{qdg}} R^2 \sqrt{\frac{H_1bc}{qd^2gg_1}}	.}

The calculation of $S_1(\mathcal I_{\mathfrak n}, \mathcal I_{\mathfrak l})$ proceeds similarly. Here we write Taylor polynomials for $\e{ \pr{-\frac{2 u}{t} \sqrt{\frac{ Y}{bcgg_1}}  +\frac{2 }{t} \sqrt{\frac{ x' Y}{bcgg_1}} } \sqrt{\frac{\ell}{n}}} $  around $x = \pr{-\frac{2 u}{t} \sqrt{\frac{ Y}{bcgg_1}}  +\frac{2 }{t} \sqrt{\frac{ x' Y}{bcgg_1}} } \sqrt{\frac{\mathfrak l}{\mathfrak n}}$ instead, and we derive that $S_1(\mathcal I_{\mathfrak n}, \mathcal I_{\mathfrak l})$ can be bounded by the same quantity. 

Next we bound $P_i(T_1, H_1; q)$. First, the length of $\mathcal I_{z, \beta}$ in Equation (\ref{lengthI}) and the fact that the number of relevant tuples is $O\pr{\frac{q^{\epsilon_1/2}}{R^3}}$ yields 
\est{\label{finalboundP2} \mathcal P_i(T_1, H_1; q) & \ll  \frac{q^{\epsilon} \sqrt {H_1}}{T_1^2 \sqrt{q} (bc)^{3/2}}\sum_{ 0 \leq \beta \ll q^{\epsilon}    \sqrt{\frac{Y}{H_1q}}} \int_0^{\infty} |\Psi_5(z)| \int_0^\infty \int_{\mathcal I_{z, \beta}} |u - \sqrt{x'}||{\Psi_7\pr{(u - \sqrt {x'})^2}}{\Psi_7\pr{x'}}| \\
	&  \times   \sum_{\substack{d| bc \\ (d,q)=1}} d \sum_{ \substack{g \leq T/d \\ (g, q) = 1}} g^{-1/2} \sum_{\substack{g_1 \leq \frac{T}{dg} \\ (g_1, q) = 1}} g_1^{-1/2}  \sumt_{(\mathcal I_{\mathfrak n}, \mathcal I_{\mathfrak n'}, \mathcal I_{\mathfrak l}, \mathcal I_{\mathfrak l'})} \pr{\frac{T_1^2}{dgg_1} + R^2 \frac{N}{g} \frac{NH_1bc}{qdg}} R^2 \sqrt{\frac{H_1bc}{qd^2gg_1}} \> du \> dx' \> dz \\
	&\ll q^{\epsilon + \epsilon_1} \frac{\sqrt {H_1}}{T_1^2 \sqrt{q} } \sqrt{\frac{Y}{H_1q}} T_1\sqrt{\frac{q}{H_1Y}}\frac{1}{R} \pr{T_1^2 + \frac{R^2N^2H_1}{q}} \sqrt{\frac{H_1}{q}} 
\ll q^{\epsilon},
}
where we recall that $R = \frac{T_1}{q^{\epsilon_1/4}} \sqrt{\frac{q}{H_1Y}} $, $N \ll \sqrt Y$, $H_1 \ll q$, $Y \leq q^{2 + \epsilon}$, and we choosing sufficiently small $\epsilon_1.$ This completes the lemma.


\end{proof}
\subsubsection*{ Calculation of $\mathcal G_{s, 2}(T_1, H_1; q)$} 
By changing variable and letting $u = \sqrt x - \sqrt{x'}$, we have that 
\est{\mathcal J_2(n, n', \ell, \ell') 
	&= \int_{0}^\infty \Psi_5(z) \int_0^\infty \int_0^\infty  \e{ -\frac{2 u}{t} \sqrt{\frac{\ell Y}{bcngg_1}}  -\frac{2 }{t} \sqrt{\frac{ x' Y}{bcgg_1}} \pr{ \sqrt{\frac{\ell}{n}} - \sqrt{\frac{\ell'}{n'}}} } \\
	&\times \e{  \frac{2u}{tdgg_1} \sqrt{\frac{z H_1Y}{q}} - \frac{\beta H_1 z}{tdgg_1}}  2(u + \sqrt{x'}){\Psi_7\pr{(u + \sqrt {x'})^2}}{\Psi_7\pr{x'}}  \>du \> dx' \> dz.}
Note that the function vanishes when $u$ is outside the range $1 \leq u + \sqrt{x'} \leq \sqrt 2$ because of the support of $\Psi_7(x)$. Further, $\mathcal J_2$ has a similar expression to $\mathcal J_1$ in Equation \ref{def:J1}. In fact, we can use the same arguments as bounding $\mathcal G_{s, 1}(T_1, H_1; q)$ to conclude that 
$$ \mathcal G_{s,2}(T_1, H_1;q) \ll q^\epsilon.$$
From the calculation of both $\mathcal G_{s, i}(T_1, H_1; q)$, we arrive at Lemma \ref{lem:F2T1small}.  From Lemma \ref{lem:F2T1big} - \ref{lem:F2T1small}, we obtain Lemma \ref{lem:boundFi} for $i = 2.$

\section{Bounding $\mathcal F_3(T_1, H_1; q)$} \label{sec:F3}

In this section, we sketch how to prove Lemma \ref{lem:boundFi} for $i = 3.$ We recall that $\mathcal F_3(T_1, H_1; q)$ is defined in Equation (\ref{def:Fi}).  This is the same expression as for $\mathcal F_2$ with $Y_0$ replaced by $K_0$.  Thus, the proof proceeds along the same lines, but we use \eqref{asymK0} in place of \eqref{asymY0} and \eqref{K0int} in place of \eqref{Y0int}.  Here, the proofs of Case 1 and Case 2 are similar.  The proof for Case 3 is significantly easier because the expression for $K_0$ in \eqref{asymK0} is a rapidly decreasing smooth function with small derivatives, as opposed to the expression for $Y_0(2\pi x)$ in \eqref{asymY0} which involves the phase $e(x - 1/8)$.  Thus, in Case 3, we can simply integrate by parts many times to show that the contribution in that range is negligible.

\appendix

\section*{Acknowledgement} 
Vorrapan Chandee would like to acknowledge the financial support from the Thailand Research Fund (TRF) under the contract number TRG5880076. She also would like to thank Kannan Soundararajan and Vichian Laohakosol for helping with TRF grant document. Part of this work was done while she was in residence at the Mathematical Sciences Research Institute (MSRI) in Berkeley, California, during the Spring semester of year 2017, supported in part by the National Science Foundation (NSF) under Grant No. DMS-1440140.

\end{document}